\documentclass[12pt,reqno]{amsart}
\usepackage{indentfirst, amssymb, amsmath, amsthm, mathrsfs, setspace, indentfirst, enumerate,  mathrsfs, amsmath, amsthm}
\usepackage[bookmarksnumbered, colorlinks, plainpages]{hyperref}
\usepackage{mathrsfs}
\usepackage{cite}
\usepackage{tikz}
\usepackage{graphicx}
\usetikzlibrary{arrows.meta,calc,decorations.pathreplacing,positioning}
\usetikzlibrary{shapes.geometric}
\usepackage{float}
\usepackage{pgfplots}
\usetikzlibrary{shadows}

\newtheorem*{theoA}{Theorem A}
\newtheorem*{theoB}{Theorem B}
\newtheorem*{theoC}{Theorem C}
\newtheorem*{theoD}{Theorem D}
\newtheorem*{theoE}{Theorem E}
\newtheorem*{theoF}{Theorem F}

\newtheorem*{cor A}{Corollary A}
\newtheorem*{cor B}{Corollary B}

\newtheorem{theo}{Theorem}[section]
\newtheorem{lem}{Lemma}[section]
\newtheorem{cor}{Corollary}[section]

\newtheorem{prob}{Problem}[section]

\newtheorem{defi}{Definition}[section]
\newtheorem{rem}{Remark}[section]

\newcommand{\ol}{\overline}
\newcommand{\be}{\begin{equation}}
\newcommand{\ee}{\end{equation}}
\newcommand{\beas}{\begin{eqnarray*}}
\newcommand{\eeas}{\end{eqnarray*}}
\newcommand{\bea}{\begin{eqnarray}}
\newcommand{\eea}{\end{eqnarray}}

\numberwithin{equation}{section}
\begin{document}
\title[M\MakeLowercase{ultidimensional analogues of the improved} B\MakeLowercase{ohr's} \MakeLowercase{inequality}....]{\LARGE M\MakeLowercase{ultidimensional analogues of the improved} B\MakeLowercase{ohr's} \MakeLowercase{inequality for shifted polydisks}}
\date{}
\author[S. M\MakeLowercase{ajumder}, A. B\MakeLowercase{anerjee}, J. B\MakeLowercase{anerjee} \MakeLowercase{and} S. P\MakeLowercase{anja} ]{S\MakeLowercase{ujoy} M\MakeLowercase{ajumder}$^1$, A\MakeLowercase{bhijit} B\MakeLowercase{anerjee}$^2$, J\MakeLowercase{hilik} B\MakeLowercase{anerjee}$^3$$^*$ \MakeLowercase {and} S\MakeLowercase{hantanu} P\MakeLowercase{anja}$^4$ }

\address{$^1$ Department of Mathematics, Raiganj University, Raiganj, West Bengal-733134, India.}
\email{sm05math@gmail.com, sjm@raiganjuniversity.ac.in}

\address{$^{2}$ Department of Mathematics, University of Kalyani, West Bengal 741235, India.}
\email{abanerjee\_kal@yahoo.co.in, abhijitbanerjee@klyuniv.ac.in}

\address{$^{3}$ Department of Mathematics, University of Kalyani, West Bengal 741235, India.}
\email{jhilikbanerjee38@gmail.com}

\address{$^4$ Department of Mathematics, University of Kalyani, West Bengal 741235, India.}
\email{panjasantu07@gmail.com, shantanumath26@klyuniv.ac.in}

\renewcommand{\thefootnote}{}
\footnote{2020 \emph{Mathematics Subject Classification}: 30A10, 30C45, 30C62, 30C75, 30H05.}
\footnote{\emph{Key words and phrases}: Bounded Holomorphic function in a  Polydisk, pluriharmonic mapping, Multidimensional Bohr inequality.}
\footnote{*\emph{Corresponding Author}: Jhilik Banerjee.}
\renewcommand{\thefootnote}{\arabic{footnote}}
\setcounter{footnote}{0}
\begin{abstract}
In this article, we investigate the Bohr phenomenon for holomorphic functions defined on a general simply connected domain in $\mathbb{C}^n$. We improve the existing results of Evdoridis et al. (Improved Bohr’s Inequality for Shifted Disks, Results in Mathematics, 76, 14 (2021), https://doi.org/10.1007/s00025-020-01325-x) for a broader class of holomorphic functions in $\mathbb{C}^n$. Furthermore, we consider pluriharmonic mappings defined on a polydisk containing the unit polydisk $\mathbb{P}\Delta(0_n;1_n)$ and establish a Bohr-type inequality for this class of mappings.
\end{abstract}
\thanks{Typeset by \AmS -\LaTeX}
\maketitle
\section{\bf Introduction}
The study of coefficient estimates for bounded analytic functions has long occupied a central position in complex analysis, owing to its deep connections with geometric function theory, operator theory, functional analysis, and multidimensional holomorphic mappings. Among the numerous results in this direction, Bohr's inequality occupies a distinguished place because it reveals an unexpected relationship between the modulus of an analytic function and the majorant of its power series. During the last century, this remarkable phenomenon has inspired a vast body of research, leading to a variety of refinements, extensions, and applications in both one and several complex variables.

One of the major themes in the modern theory of Bohr inequalities is the search for sharp Bohr radii under increasingly general geometric and analytic settings. Recent developments have demonstrated that the Bohr phenomenon extends far beyond the classical unit disk and naturally interacts with weighted coefficient problems, area estimates, subordinate families, operator-valued mappings, and holomorphic functions defined on multidimensional domains. These developments have transformed Bohr's original theorem into a fundamental principle connecting several branches of modern analysis.

To formulate the multidimensional framework considered in this paper, we first introduce the following notation.\par
\begin{figure}[t]
	\centering
	\begin{tikzpicture}[
		node distance=9mm,
		box/.style={
			rectangle,
			rounded corners=2mm,
			draw=blue!70!black,
			very thick,
			fill=blue!5,
			minimum width=8.3cm,
			minimum height=9mm,
			align=center,
			font=\small},
		arrow/.style={
			-{Stealth[length=3mm]},
			very thick,
			blue!70!black}
		]
		
		\node[box] (A)
		{\textbf{Classical Bohr Theorem (1914)}\\
			Harald Bohr established the original Bohr inequality for bounded analytic functions.};
		
		\node[box,below=of A] (B)
		{\textbf{Generalizations of Bohr's Phenomenon}\\
			Coefficient estimates, Bohr radius, extremal problems and geometric refinements.};
		
		\node[box,below=of B] (C)
		{\textbf{General Domains}\\
			Bohr inequalities for $\Omega_{\gamma}$ and other simply connected domains.};
		
		\node[box,below=of C] (D)
		{\textbf{Recent Refinements}\\
			Area versions, weighted inequalities and norm improvements.};
		
		\node[box,below=of D] (E)
		{\textbf{Several Complex Variables}\\
			Holomorphic mappings on the unit polydisc and multidimensional Bohr theory.};
		
		\node[box,below=of E,fill=green!12,draw=green!60!black] (F)
		{\textbf{Present Paper}\\
			New multidimensional Bohr-type inequalities together with sharp refinements and extensions.};
		
		\draw[arrow] (A)--(B);
		\draw[arrow] (B)--(C);
		\draw[arrow] (C)--(D);
		\draw[arrow] (D)--(E);
		\draw[arrow] (E)--(F);
		
	\end{tikzpicture}
	
	\caption{Evolution of Bohr-type inequalities leading to the present investigation.}
	\label{fig:introduction-overview}
\end{figure}
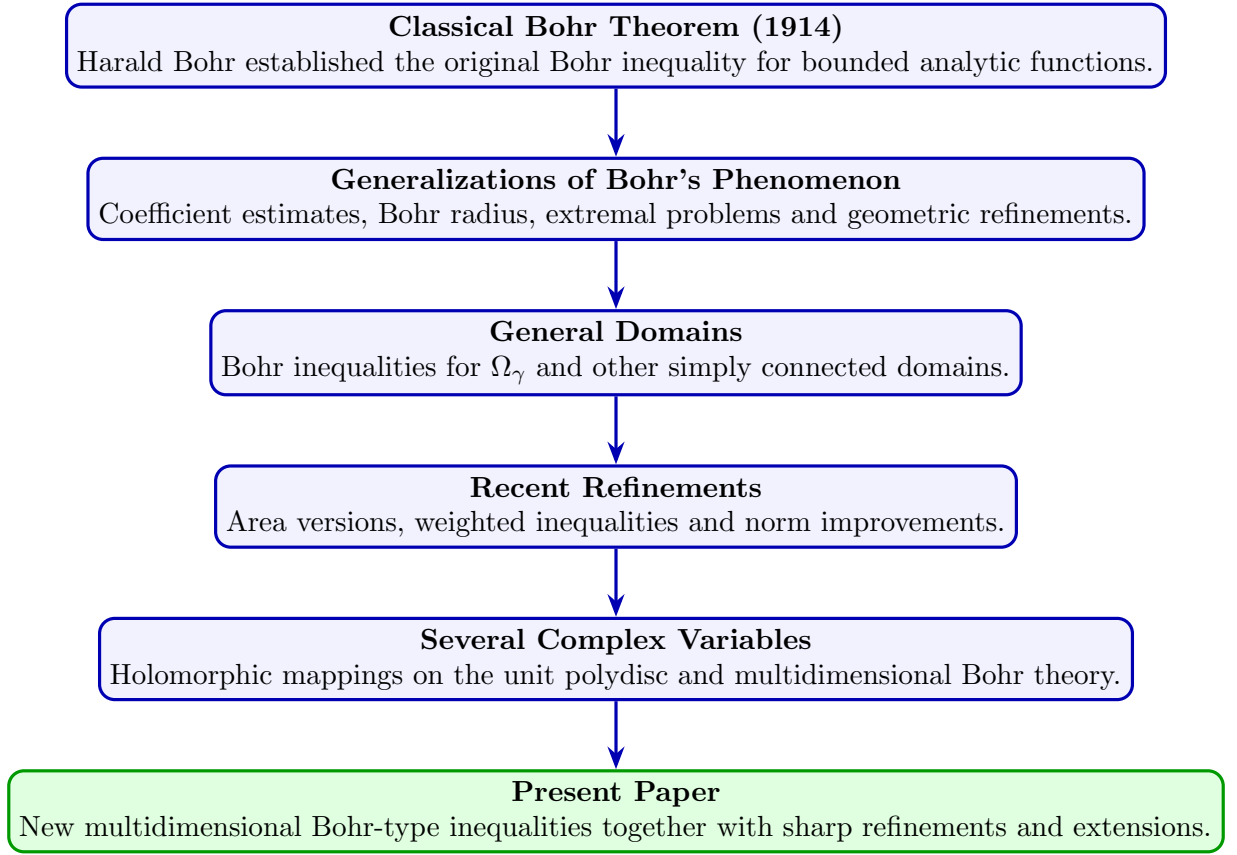

Let $\mathbb{P}\Delta(a;r)=\lbrace z\in\mathbb{C}^n: |z_i-a_i|<r_i,\;i=1,2,\ldots,n\rbrace$ denote the open polydisc in $\mathbb{C}^n$. Hrere $a=(a_1,\ldots,a_n)\in\mathbb{C}^n$ is called the centre of the polydisc and $r=(r_1,r_2,\ldots,r_n)\in\mathbb{R}^n\;(r_i>0)$ is called the polyradius. In particular, the unit polydisc in $\mathbb{C}^n$ is denoted by $\mathbb{P}\Delta(0_n;1_n)$, where $0_n=(0,0,\ldots,0)$ and $1_n=(1,1,\ldots,1)$. The unit disk in the complex plane is denoted by $\mathbb{D}$. For a given simply connected domain $\Omega_n\subset \mathbb{C}^n$ containing $\mathbb{P}\Delta(0_n;1_n)$, let
$\mathcal{H}_n(\Omega_n)$ denote the class of holomorphic functions on $\Omega_n$, and let $\mathcal{B}_n(\Omega_n)$ be the class of
functions $f\in\mathcal{H}_n(\Omega_n)$ such that $f(\Omega_n)\subseteq \ol{\mathbb{D}}$. In particular $\mathcal{H}(\Omega)(=\mathcal{H}_1(\Omega_1))$ denote the class of analytic functions on $\Omega$.\vspace{1.2mm}

A multi-index $\alpha=(\alpha_1,\ldots,\alpha_n)$ of dimension $n$ consists of n non-negative integers $\alpha_j,\;1\leq j\leq n$; the degree of a multi-index $\alpha$ is the sum $|\alpha|=\sum_{j=1}^n \alpha_j$ and we denote $\alpha!=\alpha_1!\ldots \alpha_n!$. For $z=(z_1,\ldots,z_n)\in\mathbb{C}^n$ and a multi-index $\alpha=(\alpha_1,\ldots,\alpha_n)$, we define 
\[z^{\alpha}=\prod\limits_{j=1}^n z_j^{\alpha_j}\;\;\text{and}\;\;|z|^{\alpha}=\prod\limits_{j=1}^n |z_j|^{\alpha_j}.\]\par 
 Also by $z-\gamma_n$ we mean $(z_1-\gamma,z_2-\gamma,\ldots, z_n-\gamma)$.

\smallskip
Each function $f\in \mathcal{H}_n(\mathbb{P}\Delta(0_n;1_n))$ has the following power series expansion in $z_1,\ldots,z_n$,
\begin{align}\label{In-1.1}
f(z)=\sum\limits_{\alpha_1,\alpha_2,\ldots,\alpha_n=0}^{\infty} a_{\alpha_1,\alpha_2,\ldots,\alpha_n}z_1^{\alpha_1}z_2^{\alpha_2}\ldots z_n^{\alpha_n}=
\sum\limits_{m=0}^{\infty}\sum\limits_{|\alpha|=m} a_{\alpha}z^{\alpha}=\sum\limits_{|\alpha|=0}^{\infty} P_{|\alpha|}(z),
\end{align}
which is absolutely convergent in $\mathbb{P}\Delta(0_n;1_n)$, where the term $P_k(z)$ is a homogeneous polynomial in $z_1,z_2,\ldots,z_n$ of degree $k$. For a given power series of the form \eqref{In-1.1}, its majorant series is defined by
\begin{align*}
M_f(r)=\sum\limits_{m=0}^{\infty}\sum\limits_{|\alpha|=m} |a_{\alpha}z^{\alpha}|=\sum\limits_{m=0}^{\infty}\sum\limits_{|\alpha|=m} |a_{\alpha}|r^{\alpha},
\end{align*}
where $r=(|z_1|,|z_2|,\ldots,|z_n|)=(r_1,r_2,\ldots,r_n)$ such that $r_j<1$ for $j=1,2,\ldots,n$. Clearly, for each $f\in\mathcal{B}(\mathbb{D})$, we write $M(r):=M_f(r)$ which is an increasing function of $r$ for $0\leq r<1$. Moreover, $M_f(0)=|a_0|=|f(0)|\leq 1$. It is worth noting that, for certain functions $f\in\mathcal{B}(\mathbb{D})$, the majorant series $M_f(r)$ exceeds $1$ (see \cite{Boas-2000}). This naturally leads to the following question: for which values of $r\in[0,1)$ does the inequality
\begin{align}\label{In-1.3}
M_f(r)=\sum\limits_{m=0}^{\infty}|a_m|r^m\leq 1
\end{align}
hold for every $f\in\mathcal{B}(\mathbb{D})$? The inequality \eqref{In-1.3} on $[0,1)$ is
usually known as Bohr inequality for the class $\mathcal{B}(\mathbb{D})$. 

In $1914$, Harald Bohr \cite{Bohr-1914} observed that the inequality \eqref{In-1.3} is true for $|z|\leq 1/6$ and this practice was further expanded by Wiener, Riesz, and Schur who
independently established the inequality \eqref{In-1.3} on the disk $|z|=r\leq 1/3$. Proofs
have also been given by Sidon \cite{Sidon-1927} and Tomi\'c \cite{Tomic-1962} and it is described
precisely as follows:

\begin{theoA}If $f\in \mathcal{B}(\mathbb{D})$, then $M_f(r)\leq 1$ for $0\leq r\leq 1/3$. The number $1/3$ is best possible.
\end{theoA}

The inequality $M_f(r)\leq 1$ for $f\in \mathcal{B}(\mathbb{D})$ does not hold for any $r>1/3$. This can be verified by considering the M\"{o}bius transformation
\begin{align*}
\phi_a(z)=\frac{a-z}{1-az}
\end{align*}
where $a\in (0,1)$ is chosen sufficiently close to $1$. For background information about this inequality and further work
related to Bohr's inequality, we refer the reader to the recent survey by Abu-Muhanna et al. \cite{Abu-Muhanna-Ali-Ponnusamy-2016}
and the references therein.\vspace{1.2mm}

\subsection{\bf{The Classical Bohr Inequality: Recent Implications and Applications}}

The theory of Bohr inequalities has become an important and active area of research in complex analysis, geometric function theory, and operator theory.  Since the discovery of Theorem A, Bohr's theorem has inspired extensive research leading to numerous generalizations, refinements, and applications in one and several complex variables.

Over the past century, the classical Bohr inequality has developed into a broad and influential area of research (see \cite{Alkhaleefah-Kayumov-Ponnusamy-2019, Beneteau-CMFT-2004, Bombieri-2004, Evdoridis-2019, Fournier-2010, Garcia-2018}, \cite{Kayumov-2017}-\cite{Liu- Ponnusamy-2019}, \cite{Ponnusamy-arXiv, Ponnusamy-2020} and the references therein). Numerous extensions have been established for analytic, harmonic, and meromorphic functions, including univalent, starlike, convex, Bloch, and subordinate families, together with weighted, logarithmic, and Bohr--Rogosinski variants. These developments have substantially deepened the understanding of coefficient estimates and extremal problems in geometric function theory.\vspace{1.2mm}

For $0\leq \gamma<1$, we consider the disk $\Omega_{\gamma}$ defined by
\begin{align*}
\Omega_{\gamma}=\left\lbrace z\in\mathbb{C}:\left|z+\frac{\gamma}{1-\gamma}\right|<\frac{1}{1-\gamma}\right\rbrace.
\end{align*}

It is clear that $\mathbb{D}\subset \Omega_{\gamma}$. In $2010$, Fournier and Ruscheweyh \cite{Fournier-2010} extended Bohr's inequality in the following form.

\begin{theoB}\emph{\cite[Theorem 1]{Fournier-2010}} For $0\leq \gamma<1$, let $f\in\mathcal{B}(\Omega_{\gamma})$ with $f(z)=\sum_{m=0}^{\infty}a_m z^m$ in $\mathbb{D}$. Then 
\begin{align*}
\sum\limits_{m=0}^{\infty}|a_m|r^m\leq 1\quad\;\;\text{for}\;\;r\leq \rho_{\gamma}=\frac{1+\gamma}{3+\gamma}.
\end{align*} 
Moreover, $\sum_{m=0}^{\infty}|a_m|\rho_{\gamma}^m=1$ holds for a function $f(z)=\sum_{m=0}^{\infty}a_mz^m$ in $\mathcal{B}(\Omega_{\gamma})$ if and only if $f(z)=c$ with $|c|=1$.	
\end{theoB}

Throughout this paper $S_r(f)$ denotes the area of the image of the subdisk $\mathbb{D}_r=\{z\in\mathbb{C}:|z|<r\}$ under the mapping $f$. If $f(z)=\sum_{k=0}^{\infty}a_k z^k$, then from the definition of $S_r$, we see that
\begin{align*}
\frac{S_r}{\pi}=\frac{1}{\pi}\int\int_{|z|<r} |f'(z)|^2\;dx\;dy=\sum_{k=1}k|a_k|^2r^{2k}.
\end{align*}

Recently, Evdoridis et al. \cite{Evdoridis-2021} improved Theorem B and established the following result.

\begin{theoC}\emph{\cite[Theorem 1]{Evdoridis-2021}} For $0\leq \gamma<1$, let $f\in\mathcal{B}(\Omega_{\gamma})$ with $f(z)=\sum_{m=0}^{\infty}a_m z^m$ in $\mathbb{D}$. Then we have
\begin{align*}
\sum\limits_{m=0}^{\infty}|a_m|r^m+\frac{8}{9}\left(\frac{S_{r(1-\gamma)}}{\pi}\right)\leq 1\quad \;\;\text{for}\;r\leq \frac{1+\gamma}{3+\gamma}.
\end{align*}
Moreover, the inequality is strict unless $f$ is a constant function. The bound $8/9$ and the number $(1+\gamma)/(3+\gamma)$ cannot be replaced by a larger quantity.	
\end{theoC}

For an analytic function $f(z)=\sum_{m=0}^{\infty}a_mz^m$ in $\mathbb{D}$, we write 
\begin{align*}
\|f_0\|_{r}=\sum\limits_{m=1}^{\infty}|a_m|^2r^{2m},
\end{align*}
where $f_0(z)=f(z)-f(0)$. In \cite{Ponnusamy-arXiv}, Ponnusamy et al. have shown an improvement of Theorem A: if $f\in\mathcal{B}(\mathbb{D})$, then for every $r\leq 1/3$
\begin{align*}
\sum\limits_{m=0}^{\infty}|a_m|r^m+\left(\frac{1}{1+|a_0|}+\frac{r}{1-r}\right)\|f_0\|_r\leq 1.
\end{align*}

In \cite{Evdoridis-2021}  Evdoridis et al. extended this improved version to $\mathcal{B}(\Omega_{\gamma})$ and thus obtained the following refinement of Theorem B.

\begin{theoD}\emph{\cite[Theorem 2]{Evdoridis-2021}} For $0\leq \gamma<1$, let $f\in\mathcal{B}(\Omega\mathbb{P}\Delta(0_n;1_n))$ with $f(z)=\sum_{m=0}^{\infty}a_m z^m$ in $\mathbb{D}$. Then we have
\begin{align*}
\sum\limits_{m=0}^{\infty}|a_m|r^m+\left(\frac{1}{1+|a_0|}+\frac{r}{1-r}\right)\|f_0\|_r\leq 1\quad \;\;\text{for}\;r\leq r_0= \frac{1+\gamma}{3+\gamma}
\end{align*}
and the number $r_0$ cannot be improved.	
\end{theoD}

The next result concerns a more general case, where the function under consideration is analytic in a simply connected domain $\Omega$, containing the unit disk $\mathbb{D}$. As in \cite{Fournier-2010}, we introduce
\begin{align*}
\lambda=\lambda(\Omega)=\sup\limits_{\substack{f\in\mathcal{B}(\Omega)\\n\geq 1}}\left\lbrace \frac{|a_m|}{1-|a_0|^2}: a_0\not\equiv f(z)=\sum\limits_{m=0}^{\infty}a_mz^m, \;z\in\mathbb{D}\right\rbrace.
\end{align*}

\begin{theoE}\emph{\cite[Theorem 3]{Evdoridis-2021}} Let $\Omega \supset \mathbb{D}$ be a simply connected domain and $f\in\mathcal{B}(\Omega)$ with $f(z)=\sum_{m=0}^{\infty}a_m z^m$ in $\mathbb{D}$. Then we have
\begin{align*}
B_1(r):=\sum\limits_{m=0}^{\infty}|a_m|r^m+2\left(\frac{1+\lambda}{1+2\lambda}\right)^2\frac{S_r}{\pi}\leq 1\quad \;\;\text{for}\;r\leq \frac{1}{1+2\lambda},
\end{align*}
where $S_r$ denotes the area of the image of the disk $\mathbb{D}_r$ under the mapping $f$.		
\end{theoE}

\begin{table}[H]
	\centering
	\caption{Comparison of representative Bohr-type inequalities.}
	\label{tab:comparison}
	
	\renewcommand{\arraystretch}{1.3}
	\setlength{\tabcolsep}{5pt}
	
	\begin{tabular}{|p{3.2cm}|p{2.2cm}|p{2.8cm}|p{2.6cm}|p{2.4cm}|}
		\hline
		\centering\textbf{Reference}
		&
		\centering\textbf{Domain}
		&
		\centering\textbf{Additional quantity}
		&
		\centering\textbf{Sharp radius}
		&
		\centering\textbf{Dimension}
		\tabularnewline
		\hline
		
		Bohr (1914)
		&
		$\mathbb D$
		&
		None
		&
		$\dfrac13$
		&
		1
		\tabularnewline
		\hline
		
		Fournier--Ruscheweyh (2010)
		&
		$\Omega_\gamma$
		&
		None
		&
		$\dfrac{1+\gamma}{3+\gamma}$
		&
		1
		\tabularnewline
		\hline
		
		Evdoridis \emph{et al.} (2021)
		&
		$\Omega_\gamma$
		&
		Area
		&
		$\dfrac{1+\gamma}{3+\gamma}$
		&
		1
		\tabularnewline
		\hline
		
		Evdoridis \emph{et al.} (2021)
		&
		$\Omega_\gamma$
		&
		Norm
		&
		$\dfrac{1+\gamma}{3+\gamma}$
		&
		1
		\tabularnewline
		\hline
		
		Present paper
		&
		Unit polydisc
		&
		Area + Norm + Multivariable
		&
		New sharp radius
		&
		$n$
		\tabularnewline
		\hline
		
	\end{tabular}
\end{table}

Renewed interest surged in the 1990s, driven by successful extensions to holomorphic functions of several complex variables and more abstract functional-analytic settings. For instance, in 1997, Boas and Khavinson \cite{Boas-Khavinson-PAMS-1997} introduced and determined the $n$-dimensional Bohr radius for the family of holomorphic functions bounded by unity on the polydisk. This seminal work stimulated significant research interest in Bohr-type questions across diverse mathematical domains.  Subsequent investigations have yielded further results on Bohr's phenomenon for multidimensional power series. Notable contributions in this area include those by Aizenberg \cite{Aizenberg-PAMC-1999,Aizen-PAMS-2000,Aizenberg,Aizenberg-SM-2007}, Aizenberg \textit{et al.} \cite{Aizenberg-SM-2005,Aizenberg-Aytuna-Djakov-JMAA-2001,Aigenber-CMFT-2009}, Defant and Frerick \cite{Defant-Frerick-IJM-2006}, and Djakov and Ramanujan \cite{Djakov-Ramanujan-JA-2000}. A comprehensive overview of the various aspects and generalizations of Bohr's inequality can be found in the monograph by Kresin and Maz'ya \cite{Kresin-1903}, and the references cited therein. In particular, Section 6.4 of \cite{Kresin-1903} on Bohr-type theorems highlights rich opportunities to extend several existing inequalities to holomorphic functions of several complex variables and, significantly, to solutions of partial differential equations (PDEs).  
\paragraph{Research roadmap.}
The progression of results outlined above naturally raises the question of whether the existing Bohr-type inequalities can be further refined within a broader multidimensional framework while simultaneously incorporating geometric quantities and sharper coefficient estimates. The present work addresses this question by extending and improving several known inequalities in a unified setting.
\subsection{\bf Basic Notations in several complex variables}\label{Sub-Sec-1.3}
For $z=(z_1,\ldots,z_n)$ and $w=(w_1,\ldots,w_n)$ in $\mathbb{C}^{n}$, we denote 
\begin{align*}
\langle z,w\rangle=z_1\ol w_1+\ldots+z_n \ol w_n
\end{align*}
 and $||z||=\sqrt{\langle z,z\rangle}$. The absolute value of a complex number $z_i$ is denoted by $|z_i|$ and for $z\in\mathbb{C}^n$, we define 
\begin{align*}
||z||_{\infty}=\max\limits_{1\leq i\leq n}|z_i|.
\end{align*} 

Let $\Omega_n$ be a domain in $\mathbb{C}^n$ and let
$f:\Omega_n \to \mathbb{C}^m$ be a mapping. By writing
$f=(f_1,\ldots,f_m)$ and $f_k=u_k+iv_k$, $1\le k\le m$,
where $u_k,v_k:\Omega_n\to\mathbb{R}$, we may identify $f$ with the mapping
from $\Omega_n\subset\mathbb{R}^{2n}$ into $\mathbb{R}^{2m}$ given by
\begin{align*}
(x_1,y_1,\ldots,x_n,y_n)\longmapsto
(u_1,v_1,\ldots,u_m,v_m).
\end{align*}

The mapping $f$ is called \emph{holomorphic} on $\Omega_n$ if each of its components $f_1,\ldots,f_m$ is holomorphic on $\Omega_n$. For $f=(f_1, \ldots, f_m)$ and $f_j(z)=\sum_{\alpha}a_{j, \alpha}z^{\alpha}$ for $j=1, \ldots, m$, we denote $f(z)=\sum_{\alpha}a_{\alpha}z^{\alpha}$, where $\alpha=(a_{1, \alpha}, \ldots, a_{m,\alpha})$.\vspace{2mm}

Throughout the paper, for $0\leq \gamma<1$, we consider the shifted polydisk $\mathbb{P}\Delta(a(\gamma);r(\gamma))$ defined by
\begin{align*}
\mathbb{P}\Delta(a(\gamma);r(\gamma))=\left\lbrace z\in\mathbb{C}^n: \left|z_i+\frac{\gamma}{1-\gamma}\right|<\frac{1}{1+\gamma},\;i=1,2,\ldots,n\right\rbrace,
\end{align*}
where
\begin{align*}
a(\gamma)=\left(-\frac{\gamma}{1-\gamma},-\frac{\gamma}{1-\gamma},\ldots,-\frac{\gamma}{1-\gamma}\right) \;\text{and}\; r(\gamma)=\left(\frac{1}{1-\gamma},\frac{1}{1-\gamma},\ldots,\frac{1}{1-\gamma}\right).
\end{align*}

It is natural to raise the following open problem.
\begin{prob}\label{P-1}
Can we establish the multidimensional versions of Theorems B-E?
\end{prob}

The primary goal of this article is to provide an affirmative answer to Problem \ref{P-1}. The paper is organized as follows: In Section \ref{Sec-2}, we state the main results of this paper. In Section \ref{Sec-3}, we prove several lemmas that will play a key role in establishing our main results. Section \ref{Sec-4} then presents the detailed proofs of these main results. Also in section \ref{Sec-5}, we consider pluriharmonic mappings defined on a polydisk containing the unit polydisk $\mathbb{P}\Delta(0_n;1_n)$ and establish a Bohr-type inequality for this class of mappings.
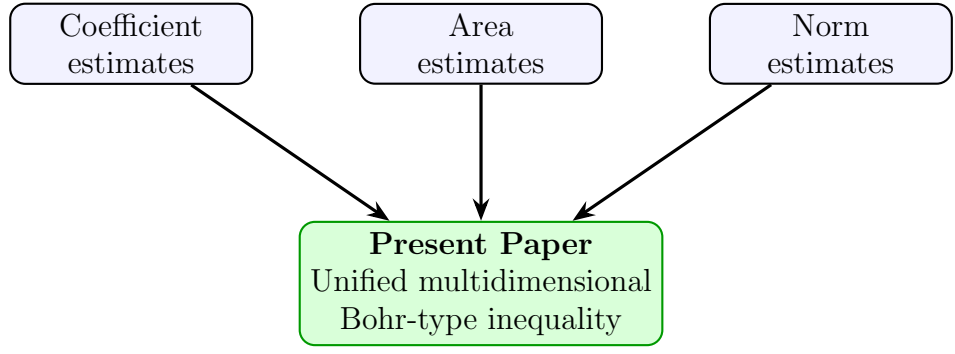
\begin{figure}[H]
	\centering
	
	\begin{tikzpicture}[
		>=Stealth,
		node distance=1.4cm,
		every node/.style={align=center},
		box/.style={
			draw,
			rounded corners=2mm,
			thick,
			fill=blue!5,
			minimum width=3.2cm,
			minimum height=0.9cm
		}
		]
		
		\node[box] (A)
		{Coefficient\\ estimates};
		
		\node[box,right=of A]
		(B)
		{Area\\ estimates};
		
		\node[box,right=of B]
		(C)
		{Norm\\ estimates};
		
		\node[box,below=1.8cm of B,
		fill=green!15,
		draw=green!60!black]
		(D)
		{\textbf{Present Paper}\\
			Unified multidimensional\\
			Bohr-type inequality};
		
		\draw[->,very thick]
		(A)--(D);
		
		\draw[->,very thick]
		(B)--(D);
		
		\draw[->,very thick]
		(C)--(D);
		
	\end{tikzpicture}
	
	\caption{Schematic illustration of the principal ingredients incorporated into the present work.}
	\label{fig:novelty}
\end{figure}
\section{{\bf Main Results}}\label{Sec-2}
We state our first result, which is a multidimensional version of Theorem C.
\begin{theo}\label{Th1} For $0\leq \gamma<1$, let $f\in\mathcal{B}(\mathbb{P}\Delta(a(\gamma);r(\gamma)))$ with $f(z)=\sum\limits_{m=0}^{\infty}\sum\limits_{|\alpha|=m}a_{\alpha} z^{\alpha}$ in $\mathbb{P}\Delta(0_n;1_n)$. Then we have
\begin{align*}
\sum\limits_{m=0}^{\infty}\sum\limits_{|\alpha|=m}|a_{\alpha}|r^{\alpha}+\frac{8}{9}\sum\limits_{m=1}^{\infty}m\sum\limits_{|\alpha|=m}|a_{\alpha}|^2r^{2\alpha}\leq 1\quad\;\;\text{for}\;\;n{\bf{r}}\leq \frac{1+\gamma}{3+\gamma},
	\end{align*} 
where $r=(|z_1|,|z_2|,\ldots,|z_n|)=(r_1,r_2,\ldots,r_n)$ such that ${\bf{r}}=\|z\|_{\infty}$.
The bound $8/9$ and the number $(1+\gamma)/(n(3+\gamma))$ cannot be replaced by a larger quantity.		
\end{theo}
\begin{figure}[H]
	\centering
	\begin{tikzpicture}
		\begin{axis}[
			width=11cm,
			height=6.5cm,
			xmin=0,
			xmax=1,
			ymin=0.30,
			ymax=0.52,
			samples=250,
			domain=0:0.999,
			axis lines=left,
			xlabel={$\gamma$},
			ylabel={Sharp radius $R(\gamma)$},
			xlabel style={font=\small},
			ylabel style={font=\small},
			tick label style={font=\small},
			enlargelimits=false,
		grid=both,
		minor tick num=3,
		major grid style={gray!40},
		minor grid style={gray!20},
			major grid style={gray!30},
			minor tick num=1,
			legend style={
				at={(0.98,0.05)},
				anchor=south east,
				draw=none,
				font=\small
			}
			]
			
			\addplot[
			blue,
			very thick
			]
			{(1+x)/(3+x)};
			
			\addlegendentry{$R(\gamma)=\dfrac{1+\gamma}{3+\gamma}$}

			\addplot[
			only marks,
			mark=*,
			mark size=2.5pt
			]
			coordinates {(0,1/3)};
			
			\node[
			anchor=north west,
			font=\small
			]
			at (axis cs:0,1/3)
			{$\left(0,\frac13\right)$};

			\addplot[
			only marks,
			mark=*,
			mark size=2.5pt
			]
			coordinates {(1,1/2)};
			
			\node[
			anchor=south east,
			font=\small
			]
			at (axis cs:1,1/2)
			{$\left(1,\frac12\right)$};
			
		\end{axis}
	\end{tikzpicture}
	
	\caption{
		Variation of the sharp Bohr radius
		$R(\gamma)=\dfrac{1+\gamma}{3+\gamma}$.
		The radius increases monotonically from $1/3$ to its limiting value $1/2$
		as $\gamma\to1^{-}$.
	}
	
	\label{fig:radius}
\end{figure}
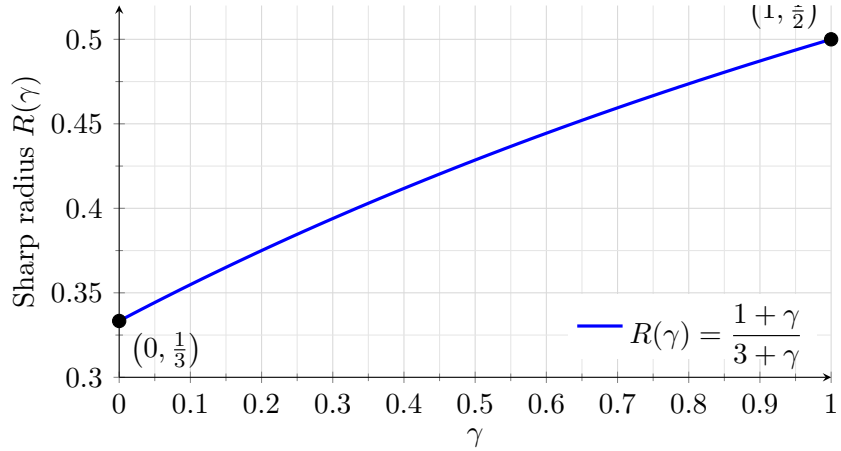
We obtain the following immediate results from Theorem \ref{Th1}.
\begin{cor}\label{cor1} For $0\leq \gamma<1$, let $f\in\mathcal{B}(\mathbb{P}\Delta(a(\gamma);r(\gamma)))$ with $f(z)=\sum\limits_{m=0}^{\infty}\sum\limits_{|\alpha|=m}a_{\alpha} z^{\alpha}$ in $\mathbb{P}\Delta(0_n;1_n)$. Then we have
\begin{align*}
\sum\limits_{m=0}^{\infty}\sum\limits_{|\alpha|=m}|a_{\alpha}|r^{\alpha}\leq 1\quad\;\;\text{for}\;\;n{\bf{r}}\leq \frac{1+\gamma}{3+\gamma},
\end{align*} 
where $r=(|z_1|,|z_2|,\ldots,|z_n|)=(r_1,r_2,\ldots,r_n)$ such that ${\bf{r}}=\|z\|_{\infty}$.
The number $(1+\gamma)/(n(3+\gamma))$ cannot be replaced by a larger quantity.		
\end{cor}

\begin{cor}\label{cor2} Let $f\in\mathcal{B}(\mathbb{P}\Delta(0_n;1_n))$ with $f(z)=\sum\limits_{m=0}^{\infty}\sum\limits_{|\alpha|=m}a_{\alpha} z^{\alpha}$. Then we have
\begin{align*}
\sum\limits_{m=0}^{\infty}\sum\limits_{|\alpha|=m}|a_{\alpha}|r^{\alpha}\leq 1\quad\;\;\text{for}\;\;n{\bf{r}}\leq \frac{1}{3},
\end{align*} 
where $r=(|z_1|,|z_2|,\ldots,|z_n|)=(r_1,r_2,\ldots,r_n)$ such that ${\bf{r}}=\|z\|_{\infty}$.
The number $1/3n$ cannot be replaced by a larger quantity.		
\end{cor}

\begin{rem}
Note that Corollary \ref{cor1} extends Theorem B from holomorphic functions of a single complex variable to those of several complex variables. Furthermore, when $\dim(\mathbb{C}^n)=n=1$, Corollary \ref{cor2} recovers the classical Bohr radius $r=1/3$. Thus, our results provide a natural generalization of the classical Bohr theorem to holomorphic functions in several complex variables.	
\end{rem}

Next, we state our second result, which is a multidimensional version of Theorem D.
\begin{theo}\label{Th2} For $0\leq \gamma<1$, let $f\in\mathcal{B}(\mathbb{P}\Delta(a(\gamma);r(\gamma)))$ with $f(z)=\sum\limits_{m=0}^{\infty}\sum\limits_{|\alpha|=m}a_{\alpha} z^{\alpha}$ in $\mathbb{P}\Delta(0_n;1_n)$. Then we have
\begin{align*}
\mathcal{M}_f(nr):=\sum\limits_{m=0}^{\infty}\sum\limits_{|\alpha|=m}|a_{\alpha}|r^{\alpha}+\left(\frac{1}{1+|a_0|}+\frac{r}{1-r}\right)\sum\limits_{m=1}^{\infty}\sum\limits_{|\alpha|=m}|a_{\alpha}|^2r^{2\alpha}\leq 1\;\text{for}\;n{\bf{r}}\leq \frac{1+\gamma}{3+\gamma},
\end{align*}
where $r=(|z_1|,|z_2|,\ldots,|z_n|)=(r_1,r_2,\ldots,r_n)$ such that ${\bf{r}}=\|z\|_{\infty}$.	The number $(1+\gamma)/(n(3+\gamma))$ cannot be improved.	
\end{theo}

Here for a simply connected domain $\Omega_n\subset\mathbb{C}^n$, containing the unit polydisk $\mathbb{P}\Delta(0_n;1_n)$, we introduce
\begin{align}\label{T.1}
\lambda_n=\lambda(\Omega_n)=\sup\limits_{\substack{f\in\mathcal{B}_n(\Omega_n)\\|\alpha|\geq 1}}\left\lbrace \frac{|a_{\alpha}|}{1-|a_0|^2}: a_0\not\equiv f(z)=\sum\limits_{m=0}^{\infty}\sum\limits_{|\alpha|=m}|a_{\alpha}|z^m, \;z\in \mathbb{P}\Delta(0_n;1_n)\right\rbrace.
\end{align}

In this section, our final result is the multidimensional generalization of Theorem E.

\begin{theo}\label{Th3} Let $\Omega_n \supset \mathbb{P}\Delta(0_n;1_n)$ be a simply connected domain in $\mathbb{C}^n$ and $f\in\mathcal{B}(\Omega_n)$ with $f(z)=\sum\limits_{m=0}^{\infty}\sum\limits_{|\alpha|=m}a_{\alpha} z^{\alpha}$ in $\mathbb{P}\Delta(0_n;1_n)$. Then we have
\begin{align*}
\mathcal{B}_f(nr):=\sum\limits_{m=0}^{\infty}\sum\limits_{|\alpha|=m}|a_{\alpha}|r^{\alpha}+2\left(\frac{1+\lambda_n}{1+2\lambda_n}\right)^2\sum\limits_{m=1}^{\infty}m\sum\limits_{|\alpha|=m}|a_{\alpha}|^2r^{2\alpha}\leq 1\;\text{for}\;n{\bf{r}}\leq \frac{1}{1+2\lambda_n}.
\end{align*}	
\end{theo}
\begin{table}[H]
	\centering
	\renewcommand{\arraystretch}{1.25}
	\begin{tabular}{|p{3.0cm}|p{2.8cm}|p{2.8cm}|p{2.6cm}|}
		\hline
		Result & Domain & Main inequality & Sharp radius\\
		\hline
		Theorem 1 &
		$\mathbb P\Delta(a(\gamma);r(\gamma))$
		&
		Bohr inequality with quadratic refinement
		&
		$\dfrac{1+\gamma}{3+\gamma}$\\
		\hline
		Theorem 2 &
		same
		&
		Improved Bohr inequality
		&
		$\dfrac{1+\gamma}{3+\gamma}$\\
		\hline
		Theorem 3 &
		General simply connected domain
		&
		Weighted Bohr inequality
		&
		$\dfrac{1}{1+2\lambda_n}$\\
		\hline
	\end{tabular}
	\vspace{1cc}
	\caption{Summary of the principal results proved in this paper.}
	\label{tab:summary}
\end{table}

\section{{\bf Auxiliary Lemmas}}\label{Sec-3}
We now recall here the following lemmas.
\begin{lem}\emph{\cite[Theorem 2.2 ]{Chen-Hamada-Ponnusamy-Vijayakumar-JAM-2024}}\label{Lem1}Let $f$ be holomorphic in the polydisk $\mathbb{P}\Delta(0_n;1_n)$ such that $|f(z)|\le 1$ for all $z\in \mathbb{P}\Delta(0_n;1_n)$. Then for all $z\in \mathbb{P}\Delta(0_n;1_n)$, we have
\begin{align*}
|f(z)|\leq \frac{|f(0)|+||z||_{\infty}}{1+|f(0)|||z||_{\infty}}.
\end{align*}
\end{lem}

The following lemma is contained in \cite[Corollary 1.3]{Liu-Chen-IJPAM-2012}.
\begin{lem}\label{Lem2} Let $f$ be holomorphic in the polydisk $\mathbb{P}\Delta(0_n;1_n)$ such that $|f(z)|\leq 1$ for all $z\in \mathbb{P}\Delta(0_n;1_n)$. Then for multi-index $\alpha=(\alpha_1,\ldots,\alpha_n)$, we have
\begin{align*}
\left|\frac{\partial^{|\alpha|} f(z)}{\partial z_1^{\alpha_1}\ldots \partial z_n^{\alpha_n}}\right|\leq \alpha!\frac{1-|f(z)|^2}{(1-||z||_{\infty}^2)^{|\alpha|}}(1+||z||_{\infty})^{|\alpha|-N}
\end{align*}
for all $z\in \mathbb{P}\Delta(0_n;1_n)$, where $N$ is the number of the indices $j$ such that $\alpha_j\neq 0$.
\end{lem}

\begin{lem}\label{Lem2a}\emph{ \cite[Lemma 2.1]{Liu-Chen-IJPAM-2012}} Let $f$ be holomorphic in the polydisk $\mathbb{P}\Delta(0_n;1_n)$ such that $|f(z)|\leq 1$ for all $z\in \mathbb{P}\Delta(0_n;1_n)$. Suppose $f(z)=a_0+\sum_{|\alpha|=1}^{\infty}a_{\alpha}z^{\alpha}$
for all $z\in \mathbb{P}\Delta(0_n;1_n)$. Then for any multi-index $\alpha$, we have 
\begin{align*}
|a_{\alpha}|\leq 1-|a_0|^2.
\end{align*}
\end{lem}

\begin{lem}\label{Lem3} Let $g$ be holomorphic in $\mathbb{P}\Delta(0_n;1_n)$ such that 
$|g(z)|\leq 1$ for all $z\in\mathbb{P}\Delta(0_n;1_n)$. Suppose
\begin{align*}
g(z)=\sum\limits_{m=0}^{\infty} \sum\limits_{|\alpha|=m} a_{\alpha}(z-\gamma)^{\alpha}
\end{align*}
for $z\in \mathbb{P}\Delta\left(\gamma;s\right)$,  where $\gamma=(\gamma_1,\gamma_2,\ldots,\gamma_n)$ and $s=(1-|\gamma_1|,1-|\gamma_2|,\ldots, 1-|\gamma_n|)$ with $|\gamma_i|<1$ for $i=1,2,\ldots,n$.
Then
\begin{align}\label{Eq-L3.1}
\mathcal{L}(r):=\sum\limits_{m=1}^{\infty}\sum\limits_{|\alpha|=m}|a_{\alpha}|r^{\alpha}+\dfrac{8}{9}\sum\limits_{m=1}^{\infty}m\sum\limits_{|\alpha|=m}|a_{\alpha}|^2r^{2\alpha}\leq 1\quad\;\text{for}\;n\|z\|_{\infty}\leq \dfrac{1-\|\gamma\|_{\infty}^2}{3+\|\gamma\|_{\infty}},
\end{align}
where $r=(|z_1|,|z_2|,\ldots,|z_n|)=(r_1,r_2,\ldots,r_n)$.
\end{lem}
\begin{proof} Without loss of generality, we may assume that $\gamma_i\in [0,1)$ for $i=1,2,\ldots,n$. 
Observe that $g(\gamma)=a_0$ and 
\begin{align}\label{Eq-L3.2}
\alpha! a_{\alpha}=\frac{\partial^{|\alpha|} g(\gamma)}{\partial z_1^{\alpha_1}\ldots \partial z_n^{\alpha_n}}.
\end{align}

Now by Lemma \ref{Lem2}, we have
\begin{align*}
\left|\frac{\partial^{|\alpha|} g(\gamma)}{\partial z_1^{\alpha_1}\ldots \partial z_n^{\alpha_n}}\right|\leq \alpha!\frac{1-|g(\gamma)|^2}{(1-||\gamma||_{\infty}^2)^{|\alpha|}}(1+||\gamma||_{\infty})^{|\alpha|-N}=&
 \alpha!\frac{1-|a_0|^2}{(1-||\gamma||_{\infty})^{|\alpha|}}\frac{1}{(1+||\gamma||_{\infty})^{N}}.
 \end{align*}
 
Since $1+||\gamma||_{\infty}\geq 1$, it follows that 
\begin{align}\label{Eq-L3.3}
\left|\frac{\partial^{|\alpha|} g(\gamma)}{\partial z_1^{\alpha_1}\ldots \partial z_n^{\alpha_n}}\right|\leq
 \alpha!\frac{1-|a_0|^2}{(1-||\gamma||_{\infty})^{|\alpha|}}\frac{1}{1+||\gamma||_{\infty}}.
 \end{align}
 
Now from \eqref{Eq-L3.2} and \eqref{Eq-L3.3}, we obtain
\begin{align}\label{Eq-L3.4}
\left|a_{\alpha}\right|\leq\frac{1-|a_0|^2}{(1-||\gamma||_{\infty})^{|\alpha|}}\frac{1}{1+||\gamma||_{\infty}}
 \end{align}
 for multi-index $\alpha=(\alpha_1,\alpha_2,\ldots,\alpha_n)$ such that $|\alpha|=m\geq 1$.
 Using \eqref{Eq-L3.4}, we deduce that
\begin{align*}
\sum\limits_{m=1}^{\infty}\sum\limits_{|\alpha|=m}|a_{\alpha}|r^{\alpha}\leq \frac{1-|a_0|^2}{1+||\gamma||_{\infty}}\sum\limits_{m=1}^{\infty}\sum\limits_{|\alpha|=m}\frac{r^{\alpha}}{(1-||\gamma||_{\infty})^{|\alpha|}}\leq \frac{1-|a_0|^2}{1+||\gamma||_{\infty}}\sum\limits_{m=1}^{\infty}\frac{(n\|z\|_{\infty})^{m}}{(1-||\gamma||_{\infty})^{m}}
\end{align*} 
and so
\begin{align}\label{Eq-L3.5}
\sum\limits_{m=1}^{\infty}\sum\limits_{|\alpha|=m}|a_{\alpha}|r^{\alpha}\leq \frac{1-|a_0|^2}{1+||\gamma||_{\infty}}\sum\limits_{m=1}^{\infty}\frac{(n\|z\|_{\infty})^{m}}{(1-||\gamma||_{\infty})^{m}}=\frac{1-|a_0|^2}{1+||\gamma||_{\infty}}\left(\frac{n\|z\|_{\infty}}{1-||\gamma||_{\infty}-n\|z\|_{\infty}}\right)
\end{align} 
provided that $n\|z\|_{\infty}<1-||\gamma||_{\infty}$.

On the other hand, using Lemma \ref{Lem2a}, we deduce that
\begin{align*}
\sum\limits_{m=1}^{\infty}m\sum\limits_{|\alpha|=m}|a_{\alpha}|^2r^{2\alpha}\leq (1-|a_0|^2)^2\sum\limits_{m=1}^{\infty}m\sum\limits_{|\alpha|=m}r^{2\alpha}\leq (1-|a_0|^2)^2\sum\limits_{m=1}^{\infty}m (n\|z\|_{\infty})^{2m}
\end{align*}
and so
\begin{align}\label{Eq-L3.6}
\sum\limits_{m=1}^{\infty}m\sum\limits_{|\alpha|=m}|a_{\alpha}|^2r^{2\alpha}\leq& (1-|a_0|^2)^2(n\|z\|_{\infty})^{2}\sum\limits_{m=1}^{\infty}(m+1) (n^2\|z\|_{\infty}^2)^{m}\\=&\frac{(1-|a_0|^2)^2(n\|z\|_{\infty})^{2}}{(1-n^2\|z\|_{\infty}^2)^2}.\nonumber
\end{align}

Let us take $\|z\|_{\infty}=\bf{r}$ and $t=\|\gamma\|_{\infty}$.
Now using \eqref{Eq-L3.5} and \eqref{Eq-L3.6}, we deduce that
\begin{align}\label{Eq-L3.6a}
\sum\limits_{m=0}^{\infty}\sum\limits_{|\alpha|=m}|a_{\alpha}|r^{\alpha}+\dfrac{8}{9}\sum\limits_{m=1}^{\infty}m\sum\limits_{|\alpha|=m}|a_{\alpha}|^2r^{2\alpha}\leq 1+\Psi(n{\bf r}),
\end{align}
where 
\begin{align}\label{Eq-L3.7}
\Psi(n{\bf{r}})=|a_0|+\frac{(1-|a_0|^2)(n\bf{r})}{(1+t)(1-t-(n\bf{r}))}+\frac{8}{9}\frac{(1-|a_0|^2)^2(n\|z\|_{\infty})^{2}}{(1-n^2\|z\|_{\infty}^2)^2}.
\end{align}
Clearly $\mathcal{L}(r)\leq 1$ provided that $\Psi(n{\bf{r}})\leq 0$. Note that
\begin{align*}
\frac{d \Psi(n\bf{r})}{d (n\bf{r})}=A+B+C+D,
\end{align*}
where
\begin{align*}
A=\frac{1-|a_0|^2}{(1+t)(1-t-n{\bf{r}})},\;B=\frac{16(1-|a_0|^2)^2(n\bf{r})}{9(1-(n\bf{r})^2)^2},
\end{align*}
\begin{align*}
C=\frac{(1-|a_0|^2)(n\bf{r})}{(1+t)(1-t-n\bf{r})^2}\;\;\text{and}\;\;D=\frac{32(1-|a_0|^2)^2(n\bf{r})^3}{9(1-(n\bf{r})^2)^3}.
\end{align*}
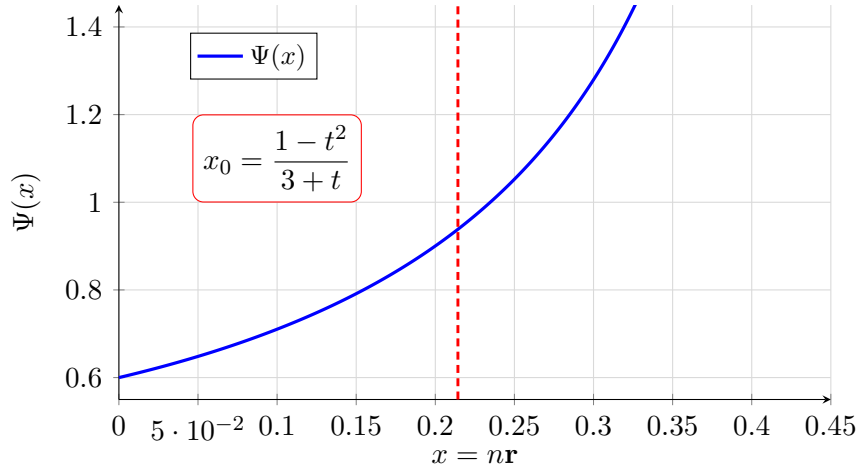
\begin{figure}[H]
	\centering
	\begin{tikzpicture}
		\begin{axis}[
			width=11cm,
			height=6.8cm,
			xmin=0,
			xmax=0.45,
			ymin=0.55,
			ymax=1.45,
			samples=300,
			domain=0:0.45,
			axis lines=left,
			xlabel={$x=n\mathbf r$},
			ylabel={$\Psi(x)$},
			grid=major,
			major grid style={gray!30},
			minor grid style={gray!15},
			tick label style={font=\small},
			label style={font=\small},
			legend style={
				font=\footnotesize,
				draw=black,
				fill=white,
				at={(0.10,0.93)},
				anchor=north west,
				inner sep=2pt
			},	every axis plot/.append style={line width=1.2pt}
			]

			\addplot[
			blue,
			smooth
			]
			{0.6
				+0.64*x/(1.5*(0.5-x))
				+(8.0/9.0)*(0.64)^2*x^2/(1-x^2)^2};
			
			\addlegendentry{$\Psi(x)$}

			\addplot[
			red,
			densely dashed,
			line width=1.1pt
			]
			coordinates
			{(0.214286,0.55) (0.214286,1.45)};
			
			\node[
			fill=white,
			draw=red,
			rounded corners,
			font=\small
			]
			at (axis cs:0.1,1.10)
			{$x_0=\dfrac{1-t^2}{3+t}$};
			
		\end{axis}
	\end{tikzpicture}
	
	\caption{
		Graph of the auxiliary function $\Psi(x)$ defined in
		\eqref{Eq-L3.7} for the representative values
		$t=\|\gamma\|_{\infty}=0.5$
		and
		$|a_0|=0.6$.
		The dashed vertical line indicates the admissible radius
		$x_0=(1-t^2)/(3+t)\approx0.214286$.
	}
	\label{fig:PsiRepresentative}
\end{figure}
By the given condition, we have $|a_0|\leq 1$. If ${n\bf{r}}<1-t$, then we see that $A\geq 0$, $B\geq 0$, $C\geq 0$ and $D\geq 0$. Consequently $\Psi({n\bf{r}})$ is an increasing function ${n\bf{r}}$ for all ${n\bf{r}}<1-t$. Since $\frac{1-t^2}{3+t}<1-t$, it follows that $\Psi({n\bf{r}})$ is an increasing function ${n\bf{r}}$ for all ${n\bf{r}}<\frac{1-t^2}{3+t}$ and so
\begin{align}\label{Eq-L3.7a}
\Psi({n\bf{r}})\leq \Psi({n\bf{r}_0})\;\;\mbox{for}\;\;n{\bf{r}}\leq {n{\bf r}_0}:=\frac{1-t^2}{3+t}.
\end{align}

To prove $\Psi({n\bf{r}})\leq 0$ it is sufficient to show that $\Psi({n\bf{r}_0})\leq 0$ for all $|a_0|\leq 1$.

Let 
\begin{align*}
\tilde A(t)=\frac{(3+t)(1-t^2)}{(3+t)^2-(1-t^2)^2}.
\end{align*}

Clearly 
\begin{align*}
	\tilde A'(t)=-\frac{t^6+6t^5+45t^2+66t+10}{(t^2-t-4)^2(t^2+t+2)^2}.
\end{align*}
Therefore $\tilde A(t)$ is a decreasing function of $t$ in $[0,1)$. Since $\tilde A(0)=3/8$ and $\tilde A(1)=0$, it follows that $\tilde A^2$ is also a decreasing function of $t$ in $[0,1]$. Hence 
\begin{align}\label{Eq-L3.8}
\tilde A^2(t)\leq \tilde A^2(0)=\frac{9}{64}.
\end{align}

Now putting ${n{\bf r}}=\frac{1-t^2}{3+t}$ into \eqref{Eq-L3.7} and then simplifying, we get 
\begin{align}\label{Eq-L3.9}
\Psi(n{\bf{r}_0})=\frac{1-|a_0|^2}{2}\left\lbrace 1+\frac{16(1-|a_0|^2)}{9}\tilde A^2(t)-\frac{2}{1+|a_0|}\right\rbrace=
\frac{1-|a_0|^2}{2}F(|a_0|),
\end{align}
where
\begin{align*}
F(x)=1+\frac{16(1-x^2)}{9}\tilde A^2(t)-\frac{2}{1+x},\;\;x\in[0,1]
\end{align*}
 We see that  
 \begin{align*}
 \lim\limits_{x\to 1^{-}}F(x)=0
\end{align*}
and
\begin{align*}
F'(x)=-\frac{32}{9}\tilde A^2(t) x+\frac{2}{(1+x)^2}=\frac{2}{(1+x)^2}\left(1-\frac{16}{9}\tilde A^2(t)x(1+x)^2\right)
 \end{align*}
 and so using \eqref{Eq-L3.8}, we get
 \begin{align*}
 F'(x)\geq \frac{2}{(1+x)^2}\left(1-\frac{x(1+x)^2}{4}\right)\geq 0.
 \end{align*}	
 
 This shows that $F(x)$ is increasing on $[0,1]$. Now if we take $|a_0|\leq x$, where $x\in [0,1]$, then we have $F(|a_0|)\leq F(x)$ and so 
 \begin{align*}
 F(|a_0|)\leq \lim\limits_{x\to 1^{-}}F(x)=0.
 \end{align*}
 Consequently from \eqref{Eq-L3.9}, we have $\Psi(n{\bf{r}_0})\leq 0$. Now from \eqref{Eq-L3.7a}, we see that $\Psi({n\bf{r}})\leq 0$ for $n{\bf{r}}\leq \frac{1-t^2}{3+t}$. Finally from \eqref{Eq-L3.6a}, we have $\mathcal{L}(r)\leq 1$ for $n{\bf{r}}\leq \frac{1-t^2}{3+t}$.
\end{proof}
\begin{figure}[H]
	\centering
	\begin{tikzpicture}[scale=1.0,line join=round,line cap=round,
>=Latex,font=\small]

\fill[
blue!2,
rounded corners=10pt
]
(-1.5,-1.6) rectangle (13.8,5.0);
\coordinate (O1) at (0,0);
\draw[->,thick] (O1)--++(3.5,0) node[right]{$z_1$};
\draw[->,thick] (O1)--++(0,3.8) node[above]{$z_n$};
\draw[->,thick] (O1)--++(-1.3,-1.3) node[left]{$z_2$};

\begin{scope}[shift={(0.55,0.55)}]

	\fill[
	blue!8,
	rounded corners=5pt
	]
	(-.18,-.18) rectangle (2.78,2.78);
	
	\filldraw[
	fill=blue!15,
	fill opacity=.35,
	draw=blue!70!black,
	very thick
	]
	(0,0)--(2,0)--(2,2)--(0,2)--cycle;
\filldraw[top color=cyan!15,
bottom color=blue!20,
draw=blue!60!black,fill opacity=.25,draw=blue!50!black,very thick]
(0,0)--(2,0)--(2,2)--(0,2)--cycle;
\filldraw[fill=blue!10,fill opacity=.25,draw=blue!70!black,very thick]
(.6,.6)--(2.6,.6)--(2.6,2.6)--(.6,2.6)--cycle;
\draw[blue!70!black,very thick](0,0)--(.6,.6);
\draw[blue!70!black,very thick](2,0)--(2.6,.6);
\draw[blue!70!black,very thick](2,2)--(2.6,2.6);
\draw[blue!70!black,very thick](0,2)--(.6,2.6);
\filldraw[
fill=yellow,
draw=blue!90!black,
line width=.5pt
]
(1.3,1.2)circle(2.2pt);
\node[right] at (1.3,1.2){$0_n$};

\end{scope}
\node[blue] at (2.5,3.8){Unit polydisk};
\node[blue] at (2.5,3.35){$\mathbb{P}\Delta(0_n;1_n)$};

\coordinate (O2) at (9,0);
\draw[
->,
>=Stealth,
line width=1.3pt,
black!80
] (O2)--++(3.5,0) node[right]{$z_1$};
\draw[
->,
>=Stealth,
line width=1.3pt,
black!80
] (O2)--++(0,3.5) node[above]{$z_n$};
\draw[
->,
>=Stealth,
line width=1.3pt,
black!80
] (O2)--++(-1.4,-1.4) node[left]{$z_2$};

\begin{scope}[shift={(9.55,0.75)}]

	\fill[
	red!8,
	rounded corners=5pt
	]
	(-.18,-.18) rectangle (2.78,2.78);
	
	\filldraw[
	fill=red!15,
	fill opacity=.35,
	draw=red!80!black,
	very thick
	]
	(0,0)--(2,0)--(2,2)--(0,2)--cycle;
\filldraw[top color=orange!20,
bottom color=red!40,
draw=red!60!black,fill opacity=.25,draw=red!60!black,very thick]
(0,0)--(2,0)--(2,2)--(0,2)--cycle;
\filldraw[fill=red!10,fill opacity=.25,draw=red!80!black,very thick]
(.6,.6)--(2.6,.6)--(2.6,2.6)--(.6,2.6)--cycle;
\draw[red!80!black,very thick](0,0)--(.6,.6);
\draw[red!80!black,very thick](2,0)--(2.6,.6);
\draw[red!80!black,very thick](2,2)--(2.6,2.6);
\draw[red!80!black,very thick](0,2)--(.6,2.6);
\filldraw[
fill=yellow,
draw=red!90!black,
line width=.5pt
]
(1.45,1.35) circle (2.2pt);
\node[right,red!80!black] at (1.45,1.35){$a(\gamma)$};

\end{scope}
\node[red!80!black] at (11.5,4.1){Shifted polydisk};
\node[red!80!black] at (11.5,3.6){$\mathbb{P}\Delta(a(\gamma);r(\gamma))$};

\draw[->,line width=1.5pt,purple!80] (4.1,1.7)--(7.5,1.7);
\node[purple!80] at (5.9,2.15){$\phi(z)=\dfrac{z-\gamma}{1-\gamma}$};

\node[draw=purple!70!black,rounded corners,fill=purple!10,
draw=purple!80,
very thick,
drop shadow,align=center]
at (5.7,.25)
{Affine map\\[2mm]
$\phi_j(z_j)=\dfrac{z_j-\gamma}{1-\gamma}$\\
$(j=1,\ldots,n)$};
\node[
font=\scriptsize,
align=center,
text=black!70
] at (6,-1) {};

\draw[
decorate,
decoration={brace,mirror,amplitude=6pt},
blue!90!black,
ultra thick
](0.55,-0.35)--(2.55,-0.35)
node[
midway,
below=7pt,
font=\scriptsize,
text=black!70
]{radius $1$};

\draw[decorate,decoration={brace,mirror,amplitude=6pt},red!90!black,
ultra thick
]
(9.55,-0.35)--(11.55,-0.35)
node[midway,below=7pt]{$r(\gamma)=\dfrac{1-\gamma}{1+\gamma}$};

\end{tikzpicture}
	\caption{
		Affine transformation from the unit polydisk
		$\mathbb P\Delta(0_n;1_n)$
		onto the shifted polydisk
		$\mathbb P\Delta(a(\gamma);r(\gamma))$.
		The cubes are symbolic representations of Cartesian products of $n$ disks; they are not geometric cubes nor drawn to scale.
	}
	\label{fig:AffineMap}
\end{figure}
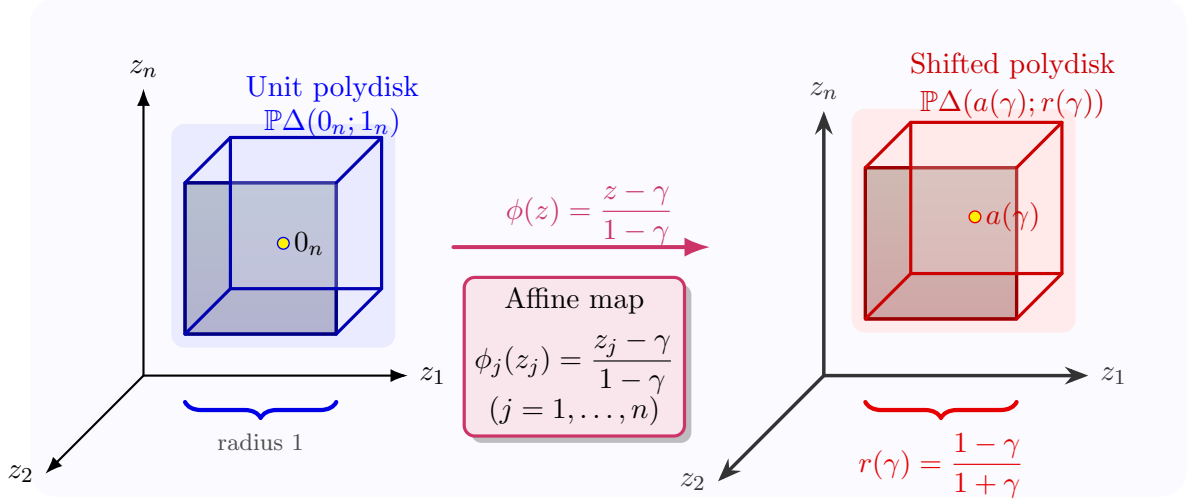

\begin{lem}\label{Lem4} Let $f$ be holomorphic in the polydisk $\mathbb{P}\Delta\left(a(\gamma);r(\gamma)\right)$ such that $|f(z)|\leq 1$ for all $z\in \mathbb{P}\Delta\left(a(\gamma);r(\gamma)\right)$. Suppose 
\begin{align*}
f(z)=a_0+\sum\limits_{m=1}^{\infty}\sum\limits_{|\alpha|=m}a_{\alpha}z^{\alpha}
\end{align*}
for all $z\in \mathbb{P}\Delta(0_n;1_n)$. Then 
\begin{align*}
|a_{\alpha}|\leq \frac{1-|a_0|^2}{1+\gamma}
\end{align*}
for multi-index $\alpha=(\alpha_1,\ldots, \alpha_n)$ such that $|\alpha|=m\geq 1$.
\end{lem}
\begin{proof} Let us define
\begin{align}\label{Eq-L4.2}
\phi(z)=\left(\frac{z_1-\gamma}{1-\gamma},\frac{z_2-\gamma}{1-\gamma},\ldots,\frac{z_n-\gamma}{1-\gamma}\right).
\end{align}
Clearly $\phi(z)\in \mathbb{P}\Delta\left(a(\gamma);r(\gamma)\right)$ if and only if $z\in  \mathbb{P}\Delta(0_n;1_n)$.
Then $g:\mathbb{P}\Delta(0_n;1_n)\to \mathbb{C}$ defined by
 \begin{align*}
 g(z)=(f\circ \phi)(z)=b_0+\sum\limits_{m=1}^{\infty}\sum\limits_{|\alpha|=m}b_{\alpha}(z-\gamma_n)^{\alpha},
 \end{align*}
 where $b_0=a_0$ and 
 \begin{align}\label{Eq-L4.3}
 b_{\alpha}=\dfrac{a_{\alpha}}{(1_n-\gamma_n)^{\alpha}}=\dfrac{a_{\alpha}}{(1-\gamma)^{|\alpha|}}
 \end{align}
 for multi-index $\alpha=(\alpha_1,\ldots, \alpha_n)$ such that $|\alpha|=m$.
 Obviously $g(z)$ is holomorphic such that $|g(z)|\leq 1$ for all $z\in\mathbb{P}\Delta(0_n;1_n)$ and $g(\gamma_n)=b_0=a_0$. Now by Lemma \ref{Lem2}, we have
 \begin{align*}
 \alpha!|b_{\alpha}|=\left|\frac{\partial^{|\alpha|} g(\gamma_n)}{\partial z_1^{\alpha_1}\ldots \partial z_n^{\alpha_n}}\right|\leq& \alpha!\frac{1-|g(\gamma_n)|^2}{(1-||\gamma_n||_{\infty}^2)^{|\alpha|}}(1+||\gamma_n||_{\infty})^{|\alpha|-N}\\=&
 	\alpha!\frac{1-|a_0|^2}{(1-\gamma^2)^{|\alpha|}}(1+\gamma)^{|\alpha|-N}
 \end{align*}
 and so from \eqref{Eq-L4.3}, we obtain
 \begin{align*}
 |a_{\alpha}|\leq
 \frac{1-|a_0|^2}{(1+\gamma)^{N}}\leq \frac{1-|a_0|^2}{1+\gamma}
 \end{align*}
 for multi-index $\alpha=(\alpha_1,\ldots, \alpha_n)$ such that $|\alpha|=m\geq 1$.
\end{proof}
\begin{figure}[H]
	\centering
	\begin{tikzpicture}[
		node distance=1.7cm,
		every node/.style={draw,rounded corners,align=center,font=\small},
		>=latex]
		
		\node (L1) {Lemma 1\\ Schwarz--Pick estimate};
		
		\node (L2) [below left=of L1]
		{Lemma 2\\Derivative estimate};
		
		\node (L3) [below right=of L1]
		{Lemma 2a\\Coefficient estimate};
		
		\node (L4) [below=2cm of L1]
		{Lemma 3\\Main Bohr inequality};
		
		\node (L5) [below=of L4]
		{Lemma 4\\Coefficient bound\\on shifted polydisk};
		
		\node (T1) [below left=2cm and 2cm of L5]
		{Theorem 1};
		
		\node (T2) [below=of L5]
		{Theorem 2};
		
		\node (T3) [below right=2cm and 2cm of L5]
		{Theorem 3};
		
		\draw[->,thick] (L1)--(L4);
		\draw[->,thick] (L2)--(L4);
		\draw[->,thick] (L3)--(L4);
		
		\draw[->,thick] (L4)--(L5);
		
		\draw[->,thick] (L5)--(T1);
		\draw[->,thick] (L5)--(T2);
		\draw[->,thick] (L5)--(T3);
		
	\end{tikzpicture}
	\caption{Logical dependence of the auxiliary lemmas and the main theorems.}
	\label{fig:roadmap}
\end{figure}
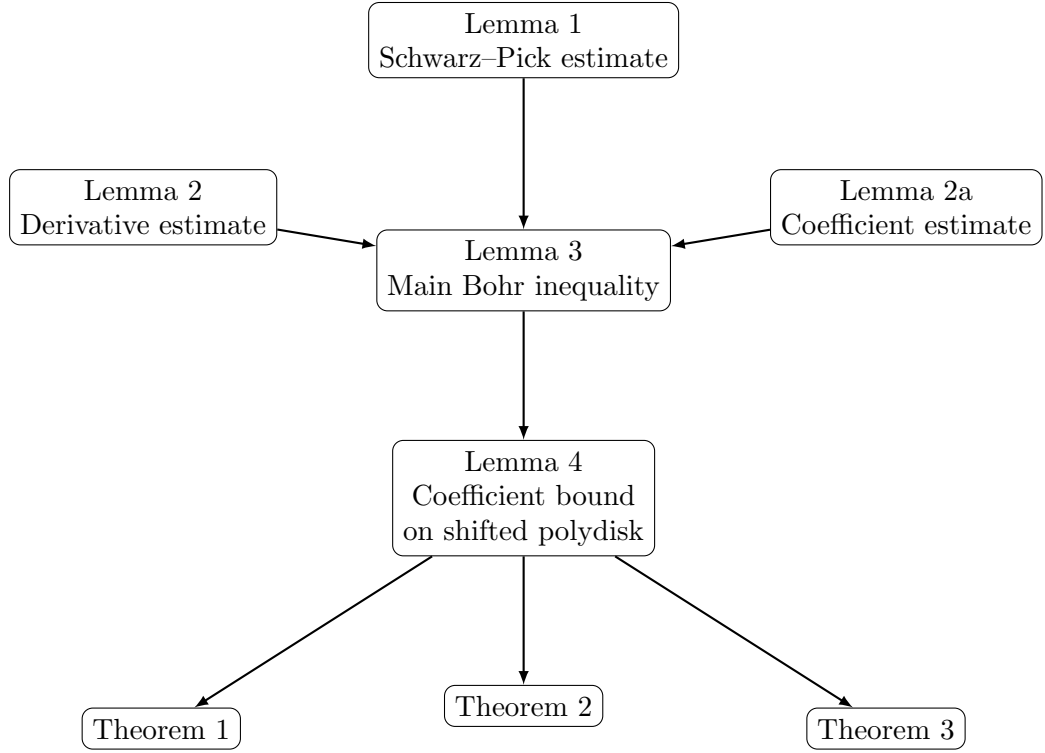
\section{\bf{Proofs of the main results}}\label{Sec-4}
\begin{proof}[\bf Proof of Theorem \ref{Th1}]
By the given condition $f$ is holomorphic in the polydisk $\mathbb{P}\Delta\left(a(\gamma);r(\gamma)\right)$ such that $|f(z)|\leq 1$ for all $z\in \mathbb{P}\Delta\left(a(\gamma);r(\gamma)\right)$. We consider the holomorphic function $\phi: \mathbb{P}\Delta(0_n;1_n)\to \mathbb{P}\Delta\left(a(\gamma);r(\gamma)\right)$ defined by \eqref{Eq-L4.2}. Then the function $g=f\circ \phi$ is holomorphic in  $\mathbb{P}\Delta(0_n;1_n)$ and 

\begin{align*}
g(z)=(f\circ \phi)(z)=\sum\limits_{m=0}^{\infty}\sum\limits_{|\alpha|=m}\frac{a_{\alpha}}{(1-\gamma)^{|\alpha|}}(z-\gamma_n)^{\alpha}
\end{align*}
for $z\in \mathbb{P}\Delta\left(\gamma_n;s\right)$,  where  $s=(1-|\gamma|,1-|\gamma|,\ldots, 1-|\gamma|)$. Now applying Lemma \ref{Lem3} for the function $g$, we get
\begin{align*}
\sum\limits_{m=0}^{\infty}\sum\limits_{|\alpha|=m}\left|\frac{a_{\alpha}}{(1-\gamma)^{\mid\alpha\mid}}\right|\rho^{\alpha}+\dfrac{8}{9}\sum\limits_{m=1}^{\infty}m\sum\limits_{|\alpha|=m}\left|\frac{a_{\alpha}}{(1-\gamma)^{\mid\alpha\mid}}\right|^2\rho^{2\alpha}\leq 1\quad\;\text{for}\;n\|\rho\|_{\infty}\leq \dfrac{1-\gamma^2}{3+\gamma},
\end{align*}
where $\rho=(\rho_1,\rho_2,\ldots,\rho_n)$. Equivalently

\begin{align}\label{Th-1.1}
\sum\limits_{m=0}^{\infty}\sum\limits_{|\alpha|=m}\frac{|a_{\alpha}|}{\left(1-\gamma\right)^{\mid\alpha\mid}}\rho^{\alpha}+\dfrac{8}{9}\sum\limits_{m=1}^{\infty}m\sum\limits_{|\alpha|=m} \frac{|a_{\alpha}|^2}{\left(1-\gamma\right)^{2\mid\alpha\mid}}\rho^{2\alpha}\leq 1\quad\;\text{for}\;\frac{n\|\rho\|_{\infty}}{1-\gamma}\leq \dfrac{1+\gamma}{3+\gamma}.
\end{align}	
Choose $\rho=((1-\gamma)|z_1|,(1-\gamma)|z_2|,\ldots, (1-\gamma)|z_n|)=((1-\gamma)r_1,(1-\gamma)r_2,\ldots, (1-\gamma)r_n)$. Then \eqref{Th-1.1} gives
\begin{align*}
	\sum\limits_{m=0}^{\infty}\sum\limits_{|\alpha|=m}|a_{\alpha}|r^{\alpha}+\dfrac{8}{9}\sum\limits_{m=1}^{\infty}m\sum\limits_{|\alpha|=m}|a_{\alpha}|^2 r^{2\alpha}\leq 1\quad\;\text{for}\;n\|z\|_{\infty}\leq \dfrac{1+\gamma}{3+\gamma}.
\end{align*}

\medskip
To prove the sharpness of the result, we consider the functions
$G:\mathbb{P}\Delta\left(a(\gamma);r(\gamma)\right)\to \mathbb{P}\Delta(0_n;1_n)$ and $\psi:\mathbb{P}\Delta(0_n;1_n)\to\mathbb{D}$ defined respectively by
\begin{align*}
G(z)=((1-\gamma)z_1+\gamma,(1-\gamma)z_2+\gamma,\ldots,(1-\gamma)z_n+\gamma)
\end{align*}
and
\begin{align*}\psi(z)=\frac{a-(z_1+z_2+\ldots+z_n)/n}{1-a(z_1+z_2+\ldots+z_n)/n}
\end{align*} 
 for $a\in(0,1)$. Then the function $g_0=\psi\circ G:\mathbb{P}\Delta\left(a(\gamma);r(\gamma)\right)\to \mathbb{D}$ takes the form
 \begin{align*}
 g_0(z)=&\frac{a-\gamma-(1-\gamma)(z_1+z_2+\ldots+z_n)/n}{1-a\gamma-a(1-\gamma)(z_1+z_2+\ldots+z_n)/n}\\=&A_0-\sum\limits_{m=1}^{\infty}A_m\left(\frac{z_1+z_2+\ldots+z_n}{n}\right)^m,
\end{align*}
where 
\begin{align}\label{Th-1.2}
A_0=\frac{a-\gamma}{1-a\gamma}\;\;\text{and}\;\;A_{m}=\frac{1-a^2}{a(1-a\gamma)}\left(\frac{a(1-\gamma)}{1-a\gamma}\right)^m.
\end{align}

Now for the point $z=(\tilde z,\tilde z,\ldots,\tilde z)$, we have 
\begin{align*}
g_0(z)=A_0-\sum\limits_{m=1}^{\infty}A_m\tilde z^m.
\end{align*}
 Suppose $a>\gamma$. For $\tilde z=nr$, provided $\left|nr+\frac{\gamma}{1-\gamma}\right|<\frac{1}{1-\gamma}$, we find that
\begin{align*}
&\sum\limits_{m=0}^{\infty}\sum\limits_{|\alpha|=m}|a_{\alpha}|(nr)^{\alpha}+\dfrac{8}{9}\sum\limits_{m=1}^{\infty}m\sum\limits_{|\alpha|=m}|a_{\alpha}|^2 (nr)^{2\alpha}\\=&
A_0+\sum\limits_{m=1}^{\infty}A_m (nr)^m+\dfrac{8}{9}\sum\limits_{m=1}^{\infty}mA^2_{m}(nr)^{2m}\\=&
\frac{a-\gamma}{1-a\gamma}+\frac{1-a^2}{1-a\gamma}\frac{(1-\gamma)(nr)}{1-a\gamma-a(1-\gamma)(nr)}+\frac{8}{9}\frac{(1-a^2)^2(1-\gamma)^4(nr)^2}{\left((1-a\gamma)^2-a^2(nr)^2(1-\gamma)^4\right)^2}\\=&1-(1-a)\Phi(nr),
\end{align*}	
where
\begin{align*}
\Phi(nr)=\frac{1+\gamma}{1-a\gamma}-\frac{1+a}{1-a\gamma}\frac{(1-\gamma)(nr)}{1-a\gamma-a(1-\gamma)(nr)}+\frac{8}{9}\frac{(1-a)(1+a)^2(1-\gamma)^4(nr)^2}{\left((1-a\gamma)^2-a^2(nr)^2(1-\gamma)^4\right)^2}.
\end{align*}
\begin{figure}[H]
	\centering
	\begin{tikzpicture}
		
	\begin{axis}[
		width=8cm,
		height=5cm,
		xmin=0,
		xmax=1.1,
		ymin=0,
		ymax=1.2,
		axis lines=left,
		thick,
		grid=both,
		minor tick num=3,
		major grid style={gray!20},
		minor grid style={gray!20},
		grid style={gray!30},
		xlabel={$t=nr$},
		ylabel={$\Phi(t)$},
		samples=200,
		domain=0:0.95,
		xtick={0,0.5,1},
		xticklabels={$0$,$\frac{1+\gamma}{3+\gamma}$,$1$},
		tick style={thick},
		legend style={
			draw=none,
			fill=none,
			at={(0.98,0.98)},
			anchor=north east
		}
		]
			
		\addplot[
		very thick,
		blue
		]
		{1-x};
		
		\draw[
		red,
		thick,
		dashed
		]
		(axis cs:0.5,-0.4)
		--
		(axis cs:0.5,1.1);
		
		\addplot[
		only marks,
		mark=*,
		mark size=2.5pt,
		red
		]
		coordinates {(0.5,0.5)};
		
		\end{axis}

	\end{tikzpicture}
	
	\caption{Typical qualitative behaviour of the decreasing function
		$\Phi(t)$.}
\end{figure}
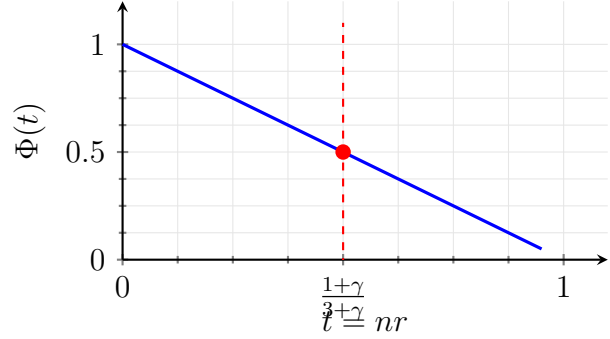

Clearly, $\Phi(nr)$ is a strictly decreasing function of $nr$ in $(0,1)$ and so for $nr>nr_0=(1+\gamma)/(3+\gamma)$, we have 
\begin{align*}
\Phi(nr)<\Phi(nr_0)=&\frac{1+\gamma}{1-a\gamma}-\frac{1+a}{1-a\gamma}\;\frac{(1-\gamma^2)}{(1-a\gamma)(3+\gamma)-a(1-\gamma^2)}\\+&\frac{8}{9}(1-a)\frac{(1+a)^2(1-\gamma)^4(1+\gamma)^2(3+\gamma)^2}{\left((1-a\gamma)^2(3+\gamma)^2-a^2(1+\gamma)^2(1-\gamma)^4\right)^2}\to 0
\end{align*}
as $a\to 1$. Consequently $\Phi(nr)$ is negative for $nr>nr_0=(1+\gamma)/(3+\gamma)$ and so $1-(1-a)\Phi(nr)>1$. This shows that the bound $8/9$ and the number $(1+\gamma)/(n(3+\gamma))$ cannot be replaced by a larger quantity.
\end{proof}

\begin{proof}[\bf Proof of Theorem \ref{Th2}]
Without loss of generality, we may assume that $a_0:=a\in (0, 1)$. By using Lemma \ref{Lem4}, we obtain
\begin{align}\label{Th-2.0}
\mathcal{M}_f(nr)=&\sum\limits_{m=0}^{\infty}\sum\limits_{|\alpha|=m}|a_{\alpha}|r^{\alpha}+\left(\frac{1}{1+|a_0|}+\frac{\bf{r}}{1-\bf{r}}\right)\sum\limits_{m=0}^{\infty}\sum\limits_{|\alpha|=m}|a_{\alpha}|^2r^{2\alpha}\\\leq &
a+\frac{1-a^2}{1+\gamma}\sum\limits_{m=1}^{\infty} (n{\bf{r}})^m+
\left(\frac{1}{1+|a_0|}+\frac{n\bf{r}}{1-n\bf{r}}\right)\left(\frac{1-a^2}{1+\gamma}\right)^2\sum\limits_{m=1}^{\infty}(n{\bf{r}})^{2m}\nonumber\\=&
a+\frac{1-a^2}{1+\gamma}\frac{n\bf{r}}{1-n\bf{r}}+
\left(\frac{1}{1+|a_0|}+\frac{n\bf{r}}{1-n\bf{r}}\right)\left(\frac{1-a^2}{1+\gamma}\right)^2\frac{(n\bf{r})^2}{1-(n\bf{r})^2}\nonumber\\=&U(a)\nonumber\\=&
a+A(1-a^2)+B(1-a)(1-a^2)+C(1-a^2)^2,\nonumber
\end{align} 
where
\begin{align*}
A=\frac{1}{(1+\gamma)}\frac{n\bf{r}}{1-n{\bf{r}}},\;\;B=\frac{1}{(1+\gamma)^2}\;\frac{(n{\bf{r}})^2}{1-(n\bf{r})^2}\;\;\mbox{and}\;\;C=\frac{1}{(1+\gamma)^2} \frac{(n{\bf{r}})^3}{(1-n{\bf{r}})(1-(n{\bf{r}})^2)}.
\end{align*}

Now for $a\in[0,1]$, we see that
\begin{align*}
U'(a)=1-2aA+B(3a^2-2a-1)+4C(a^3-a),
\end{align*}
\begin{align*}
U''(a)=-2A+2B(3a-1)+4C(3a^2-1)\;\;\mbox{and}\;\;U'''(a)=6B+24aC.
\end{align*}
\begin{table}[H]
	\centering
	\caption{Summary of the monotonicity argument.}\vspace{.5cc}
	\renewcommand{\arraystretch}{1.2}
	\begin{tabular}{|c|c|c|c|}
		\hline
		Quantity & Sign & Consequence & Result\\
		\hline
		$U'''(a)$ & $\ge0$ &
		$U''(a)$ increasing &
		$U''(a)\le U''(1)$\\
		\hline
		$U''(a)$ & $\le0$ &
		$U'(a)$ decreasing &
		$U'(a)\ge U'(1)$\\
		\hline
		$U'(a)$ & $\ge0$ &
		$U(a)$ increasing &
		$U(a)\le U(1)=1$\\
		\hline
	\end{tabular}
	\label{tab:monotonicity}
\end{table}
Since $B\geq 0$ and $C\geq 0$, it follows that $U'''(a)\geq 0$ and so $U''(a)$ is an increasing function of $a$ in $[0,1]$. Consequently 
\begin{align}\label{Th-2.1}
U''(a)\leq U''(1)=-2A+4B+8C=\frac{2(n\bf{r})}{(1+\gamma)^2(1-n\bf{r})(1-(n\bf{r})^2)}\Psi(n{\bf{r}}),
\end{align}
where
\begin{align*}
\Psi({n{\bf{r}}})=4({n\bf{r}})^2+2({n\bf{r}})(1-{n\bf{r}})-(1+\gamma)(1-({n\bf{r}})^2)=(1+{n\bf{r}})({n\bf{r}}(3+\gamma)-(1+\gamma)).
\end{align*}

For $n{\bf{r}}<n{\bf{r}_0}=(1+\gamma)/(3+\gamma)$, we see that $\Psi({n{\bf{r}}})\leq 0$ and so from \eqref{Th-2.1}, we have $U''(a)\leq 0$ in $[0,1]$ and so $U'(a)$ is a decreasing function of $a$ in $[0,1]$. Therefore 
\begin{align*}
U'(a)\geq U'(1)=1-2A=\frac{1+\gamma-{n\bf{r}}(3+\gamma)}{(1+\gamma)(1-{n{\bf{r}}})}.
\end{align*}

For $n{\bf{r}}<n{\bf{r}_0}=(1+\gamma)/(3+\gamma)$, we see that $U'(a)\geq 0$ in $[0,1]$ and so $U(a)$ is an increasing function of $a$ in $[0,1]$. Consequently $U(a)\leq U(1)=1$ and so from \eqref{Th-2.0}, we have $\mathcal{M}_f({n\bf{r}})\leq 1$ for $n{\bf{r}}<n{\bf{r}_0}=(1+\gamma)/(3+\gamma)$.
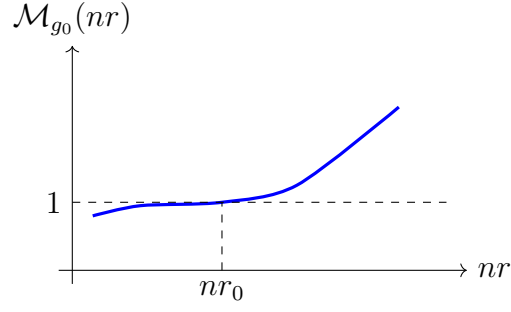
\begin{figure}[H]
	\centering
	\begin{tikzpicture}[scale=0.9]
		
		\draw[->]
		(-0.2,0)--(5.8,0)
		node[right]{$nr$};
		
		\draw[->]
		(0,-0.2)--(0,3.3)
		node[above]{$\mathcal M_{g_0}(nr)$};
		
		\draw[very thick,blue]
		plot[smooth]
		coordinates{
			(0.3,0.8)
			(1.0,0.95)
			(2.2,1.0)
			(3.3,1.25)
			(4.8,2.4)
		};
		
		\draw[dashed]
		(2.2,0)--(2.2,1);
		
		\node[below]
		at (2.2,0)
		{$nr_0$};
		
		\draw[dashed]
		(0,1)--(5.5,1);
		
		\node[left]
		at (0,1)
		{$1$};
		
	\end{tikzpicture}
	
	\caption{Sharpness: $\mathcal M_{g_0}(nr)$ exceeds $1$
		whenever $nr>\dfrac{1+\gamma}{3+\gamma}$.}
\end{figure}
\medskip
For the sharpness of the radius, we consider the function
\begin{align*}
	g_0(z)=&\frac{a-\gamma-(1-\gamma)(z_1+z_2+\ldots+z_n)/n}{1-a\gamma-a(1-\gamma)(z_1+z_2+\ldots+z_n)/n}\\=&A_0-\sum\limits_{m=1}^{\infty}A_m\left(\frac{z_1+z_2+\ldots+z_n}{n}\right)^m,
\end{align*}
where $A_m(m\geq 0)$ are given by \eqref{Th-1.2}.

Now for the point $z=(\tilde z,\tilde z,\ldots,\tilde z)$, we have 
\begin{align*}
	g_0(z)=A_0-\sum\limits_{m=1}^{\infty}A_m\tilde z^m.
\end{align*}
Suppose $a>\gamma$. For $\tilde z=nr$, where $\left|nr+\frac{\gamma}{1-\gamma}\right|<\frac{1}{1-\gamma}$, we find that
\begin{align*}
\mathcal{M}_{g_0}(nr)=&A_0+\sum\limits_{m=1}^{\infty}A_m(nr)^m+\left(\frac{1}{1+|A_0|}+\frac{nr}{1-nr}\right)\sum\limits_{m=1}^{\infty}A^2_m(nr)^{2m}\\=&
\frac{a-\gamma}{1-a\gamma}+\frac{1-a^2}{a(1-a\gamma)}\sum\limits_{m=1}^{\infty} \left(\frac{a(1-\gamma)}{1-a\gamma}\right)^m (nr)^m\\&+\left(\frac{1-a\gamma}{(1+a)(1-\gamma)}+\frac{nr}{1-nr}\right)\left(\frac{1-a^2}{a(1-a\gamma)}\right)^2\sum\limits_{m=1}^{\infty}\left(\frac{a(1-\gamma)}{1-a\gamma}\right)^{2m}(nr)^{2m}\\=&
\frac{a-\gamma}{1-a\gamma}+\frac{1-a^2}{1-a\gamma}\frac{(1-\gamma)nr}{1-a\gamma-a(1-\gamma)nr}\\&+\left(\frac{1-a\gamma}{(1+a)(1-\gamma)}+\frac{nr}{1-nr}\right)\left(\frac{1-a^2}{1-a\gamma}\right)^2\frac{(1-\gamma)^2n^2r^2}{(1-a\gamma)^2-a^2(1-\gamma)^2n^2r^2}
\\=&1-\frac{1-a}{1-a\gamma}\Phi(nr),
\end{align*}
where
\begin{align}\label{Th-2.2}
\Phi(nr)=&1+\gamma-\frac{(1+a)(1-\gamma)nr}{1-a\gamma-a(1-\gamma)nr}\\&-\left(\frac{1-a\gamma}{(1+a)(1-\gamma)}+\frac{nr}{1-nr}\right)\frac{(1+a)(1-a^2)}{1-a\gamma}\frac{(1-\gamma)^2n^2r^2}{(1-a\gamma)^2-a^2(1-\gamma)^2n^2r^2}.\nonumber
\end{align}

Note that 
\begin{align*}
\lim\limits_{a\to 1}\left(1+\gamma-\frac{(1+a)(1-\gamma)nr}{1-a\gamma-a(1-\gamma)nr}\right)=1+\gamma-\frac{2nr}{1-nr}<0
\end{align*}
if $n{\bf{r}}>n{\bf{r}_0}=(1+\gamma)/(3+\gamma)$. Consequently from \eqref{Th-2.2}, we conclude that $\Phi(nr)$ is negative for $n{\bf{r}}>n{\bf{r}_0}=(1+\gamma)/(3+\gamma)$ as $a$ tends to $1$ and hence $\mathcal{M}_{g_0}(nr)>1$ for $n{\bf{r}}>n{\bf{r}_0}=(1+\gamma)/(3+\gamma)$. This shows that the number $(1+\gamma)/(n(3+\gamma))$ cannot be replaced by a larger quantity.
\end{proof}

\begin{proof}[\bf Proof of Theorem \ref{Th3}]
By the definition of $\lambda_n$ given by (\ref{T.1}), we have
\begin{align}\label{Th-3.1}
|a_{\alpha}|\leq \lambda_n (1-|a_0|^2)
\end{align}
for multi-index $\alpha=(\alpha_1,\alpha_2,\ldots,\alpha_n)$ such that $|\alpha|\geq 1$.
Now using \eqref{Th-3.1}, we find that
\begin{align*}
\mathcal{B}_f({n\bf{r}})=&\sum\limits_{m=0}^{\infty}\sum\limits_{|\alpha|=m}|a_{\alpha}|r^{\alpha}+2\left(\frac{1+\lambda_n}{1+2\lambda_n}\right)^2\sum\limits_{m=1}^{\infty}m\sum\limits_{|\alpha|=m}|a_{\alpha}|^2r^{2\alpha}\\\leq&
|a_0|+\lambda_n (1-|a_0|^2)\frac{n\bf{r}}{1-n\bf{r}}+2\left(\frac{1+\lambda_n}{1+2\lambda_n}\right)^2\lambda_n^2(1-|a_0|^2)^2\frac{(n\bf{r})^2}{(1-(n\bf{r})^2)^2}.
\end{align*}

Now for ${n\bf{r}}<\frac{1}{1+2\lambda_n}$, we have
\begin{align*}
\mathcal{B}_f({n\bf{r}})\leq&
|a_0|+\lambda_n (1-|a_0|^2)\frac{1/(1+2\lambda_n)}{1-1/(1+2\lambda_n)}\\&+
2\left(\frac{1+\lambda_n}{1+2\lambda_n}\right)^2\lambda_n^2(1-|a_0|^2)^2\frac{\left(1/(1+2\lambda_n)\right)^2}{\left(1-\left(1/(1+2\lambda_n)\right)^2\right)^2}\\=&
1-(1-|a_0|)\left(1-\frac{1+|a_0|}{2}-\frac{(1+|a_0|)(1-|a_0|^2)}{8}\right)\\=&
1-\frac{1-|a_0|^2}{8}F(|a_0|),
\end{align*}
where
\begin{align*}
F(x)=\frac{8}{1+x}-5+x^2\;\;\text{for}\;\;x\in [0,1].
\end{align*}
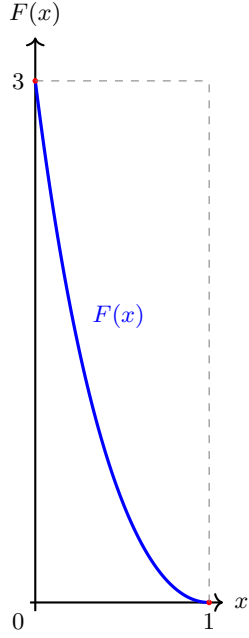
\begin{figure}[H]
	\centering
	\begin{tikzpicture}[scale=2.3]

		\draw[->,thick] (-0.03,0)--(1.08,0)
		node[right,font=\scriptsize]{$x$};
		
		\draw[->,thick] (0,-0.05)--(0,3.25)
		node[above,font=\scriptsize]{$F(x)$};

		\draw[
		blue,
		very thick,
		domain=0:1,
		samples=250,
		smooth
		]
		plot(\x,{8/(1+\x)-5+\x*\x});

		\fill[red] (0,3) circle(0.45pt);
		\fill[red] (1,0) circle(0.45pt);

		\draw[dashed,gray] (0,3)--(1,3);
		\draw[dashed,gray] (1,0)--(1,3);

		\node[left,font=\scriptsize] at (0,3){$3$};
		\node[below,font=\scriptsize] at (1,0){$1$};
		\node[below left,font=\scriptsize] at (0,0){$0$};

		\node[blue,font=\scriptsize] at (0.48,1.65){$F(x)$};
		
	\end{tikzpicture}
	\caption{Graph of $F(x)=\dfrac{8}{1+x}-5+x^2$ on $[0,1]$.}
	\label{fig:F}
\end{figure}

We see that $F(1)=0$ and $F'(x)\leq 0$ in $[0,1]$. Consequently $F(x)$ is a decreasing function of $x$ in $[0,1]$ and so $F(x)\geq F(1)=0$ in $[0,1]$. Hence, $\mathcal{B}_f({n\bf{r}})\leq 1$ for ${n\bf{r}}<\frac{1}{1+2\lambda_n}$.
\end{proof}

\section{\bf{Bohr phenomenon for Pluriharmonic mapping in shifted polydosk}}\label{Sec-5}
The connection between pluriharmonic mappings and the Bohr radius represents a highly active extension of classical complex analysis into higher dimensions (multivariable complex analysis) and functional analysis. In modern geometric function theory, studying the pluriharmonic mapping in the view of the Bohr radius means determining how the addition of an anti-holomorphic component $\overline{g(z)}$ and the geometric expansion into higher-dimensional spaces distort the classical $1/3$ boundary. Rather than a single fixed number, it manifests as a framework of Bohr-type inequalities heavily influenced by operator theory and Banach space geometry.

\subsection{\bf{Basic ideas of pluriharmonic mapping in $\Omega_n$}}
We write $z_j = x_j + i y_j\;(i^2 = -1,\; j = 1,\dots,n)$, where $x_j$ and $y_j$ are real numbers. We set
$f(z) = u(x,y) + i v(x,y)$, where $u(x,y)$ and $v(x,y)$ are the real and imaginary parts of $f(z)$; $x = (x_1,\dots,x_n)$ and $y = (y_1,\dots,y_n)$.
The Cauchy--Riemann equations for each $z_j\;(j=1,\dots,n)$ are
\begin{align}\label{Eq 1.1}
\frac{\partial u}{\partial x_j}
=
\frac{\partial v}{\partial y_j}\quad \text{and}\quad
\frac{\partial u}{\partial y_j}
=
-\,\frac{\partial v}{\partial x_j}
\qquad (j=1,\dots,n).
\end{align}

By differentiating (\ref{Eq 1.1}) with respect to $x_k$ and $y_k$, we see that both $u$ and $v$ satisfy the
following system of partial differential equations of second order:
\begin{align}\label{Eq 1.2}
\frac{\partial^2}{\partial x_j \partial x_k}
+
\frac{\partial^2}{\partial y_j \partial y_k}
= 0
\quad \text{and} \quad
\frac{\partial^2}{\partial x_j \partial y_k}
-
\frac{\partial^2}{\partial x_k \partial y_j}
= 0
\qquad (j,k = 1,\dots,n).
\end{align}

For a complex variable $z_j = x_j + i y_j$, we define
\begin{align}\label{Eq 1.3}
\frac{\partial}{\partial z_j}
= \frac{1}{2}\left( \frac{\partial}{\partial x_j}
- i \frac{\partial}{\partial y_j} \right)\;\;\text{and}\;\;
\frac{\partial}{\partial \bar z_j}
= \frac{1}{2}\left( \frac{\partial}{\partial x_j}
+ i \frac{\partial}{\partial y_j} \right).
\end{align}

We know that a function $f$ defined on an open subset $U\subset \mathbb{R}^n$ is said to be of $C^k$-class if $f$ is $k$-times continuously differentiable.

Let $f(z) = u(x,y) + i v(x,y)\;(x = (x_1,\dots,x_n), y = (y_1,\dots,y_n))$, where both $u$ and $v$ are of $C^2$-class. A direct calculation on (\ref{Eq 1.3}) shows that 
\begin{align}\label{Eq 1.5}
4\frac{\partial^2 f(z)}{\partial \bar z_j \partial z_k}=&\frac{\partial^2 u(x,y)}{\partial x_j x_k}+\frac{\partial^2 u(x,y)}{\partial y_j y_k}+i\left(\frac{\partial^2 v(x,y)}{\partial x_j x_k}+\frac{\partial^2 v(x,y)}{\partial y_j y_k}\right)\\&-i\left(\frac{\partial^2 u(x,y)}{\partial x_j y_k}-\frac{\partial^2 u(x,y)}{\partial x_k y_j}\right)+\left(\frac{\partial^2 v(x,y)}{\partial x_j y_k}-\frac{\partial^2 v(x,y)}{\partial x_k y_j}\right).\nonumber
\end{align}

A function $f:\Omega_n \to \mathbb{C}$ defined on an open set $\Omega_n\subset \mathbb{C}^n$ is said to be holomorphic if $f$ is of $C^1$-class and satisfies
\begin{align}\label{Eq 1.4}
\frac{\partial f(z)}{\partial \bar z_j}=0\quad \text{on}\;\;\Omega_n\;\;\text{ for all j}.
\end{align}

A real-valued function $\phi(x,y)$, where $x = (x_1,\dots,x_n)$ and $y = (y_1,\dots,y_n)$ is \emph{pluriharmonic} if it satisfies the conditions  (\ref{Eq 1.2}). Thus a continuous complex-valued function $f(z)=u(x,y)+iv(x,y)$, where
$x = (x_1,\dots,x_n)$ and $y = (y_1,\dots,y_n)$ is a complex-valued \emph{pluriharmonic} function in a domain $\Omega_n\subset \mathbb{C}^n$, if both $u(x,y)$ and $v(x,y)$ are real-valued \emph{pluriharmonic} functions in $\Omega_n$.
If $u(x,y)$ and $v(x,y)$, where $x = (x_1,\dots,x_n)$ and $y = (y_1,\dots,y_n)$ satisfy (\ref{Eq 1.1}), then we call $v(x,y)$, a \emph{pluriharmonic conjugate} of $u(x,y)$.

\par Thus for functions $f(z)=u(x,y)+iv(x,y)$, where
$x = (x_1,\dots,x_n)$ and $y = (y_1,\dots,y_n)$ with continuous second order partial derivatives, it is clear from (\ref{Eq 1.5}) and (\ref{Eq 1.4}) that $\frac{\partial f(z)}{\partial z_j}$ is holomorphic on $\Omega$ for all $j$ if  $f(z)$ is pluriharmonic function.
\vspace{1.5mm}
\par In a simply connected domain $\Omega_n\subset \mathbb{C}^n$, let $f(z)$ be a complex-valued pluriharmonic function. We recall that $\frac{\partial f(z)}{\partial z_j}$ is holomorphic on $\Omega_n$ for all $j=1,2,\ldots,n$ if $f(z)$ is pluriharmonic and let $\frac{\partial h(z)}{\partial z_j}=\frac{\partial f(z)}{\partial z_j}\;(j=1,2,\ldots,n)$, where $h(z)$ is holomorphic in $\Omega_n$. Now let $g(z)=\ol {f(z)}-\ol {h(z)}$ and we observe that
\[\frac{\partial g(z)}{\partial \ol{z_j}}=\ol{\frac{\partial f(z)}{\partial z_j}}-\ol{\frac{\partial h(z)}{\partial z_j}}=0\quad \text{in}\;\Omega,\quad j=1,2,\ldots,n\]
by the definition of $h$. Thus $g(z)$ is holomorphic in $\Omega_n$. Therefore the pluriharmonic function $f(z)$ has the representation $f(z)=h(z)+\ol{g(z)}$, where $h(z)$ and $g(z)$ are holomorphic in $\Omega_n$. \vspace{1.2mm}

\begin{defi} Let $f\in\mathcal{B}(\Omega_n)$. The total derivative $D(f)$ of $f$ is defined by
\begin{align*}
D(f(z))=\sum\limits_{j=1}^nz_j\frac{\partial f(z)}{\partial z_j}.
\end{align*}
\end{defi}
\medskip

\subsection{\bf {Preliminary results}}
Harmonic version of Bohr's inequality was discussed by Kayumov et al. in \cite{Kayumov-Ponnusamy-Shakirov-2018}. In a related development, Kayumov and Ponnusamy \cite{Kayumov-2018} gave several improved versions of Bohr's inequality. For generalizations of the theorem of Bohr to harmonic functions and pluriharmonic mappings, see, for example, \cite{Chen-Hamada-Ponnusamy-Vijayakumar-JAM-2024, Hamada-AAMP-2025, Kayumov-Ponnusamy-2018, Kayumov-Ponnusamy-AASFM-2019, Liu-Ponnusamy-2023, Muhanna-2010} and the references therein.\vspace{1.2mm}

In the following, Evdoridis et al. \cite{Evdoridis-2021} considered a harmonic mapping in $\Omega_{\gamma}$ and obtained Bohr's inequality for its restriction to the unit disk $\mathbb{D}$.

\begin{theoF}\emph{\cite[Theorem 4]{Evdoridis-2021}} Let $f=h+\ol g$ be a harmonic mapping in $\Omega_{\gamma}$ with $|h(z)|\leq 1$ on $\Omega_{\gamma}$. If $h(z)=\sum_{m=0}^{\infty} a_mz^m$ and $g(z)=\sum_{m=1}^{\infty} b_mz^m$ in $\mathbb{D}$ and $|g'(z)|\leq k|h'(z)|$ for some $k\in [0,1]$, then
\begin{align*}
\sum\limits_{m=0}^{\infty}|a_{m}|r^m+\sum\limits_{m=1}^{\infty}|b_{m}|r^m\leq 1\;\;\text{for}\;\;r\leq r_0:=\frac{1+\gamma}{3+2k+\gamma}.
\end{align*}
The radius $r_0$ is the best possible.
\end{theoF}

\subsection{\bf {Main result}} We state our result, which is a multidimensional version of Theorem F.
\begin{theo}\label{Th4} Let $f=h+\ol g$ be a pluriharmonic mapping in $\mathbb{P}\Delta(a(\gamma);r(\gamma))$ with $|h(z)|\leq 1$ on $\mathbb{P}\Delta(a(\gamma);r(\gamma))$. If 
\begin{align*}
h(z)=\sum\limits_{m=0}^{\infty}\sum\limits_{|\alpha|=m} a_{\alpha}z^{\alpha}\;\;\text{and}\;\; g(z)=\sum\limits_{m=1}^{\infty}\sum\limits_{|\alpha|=m} b_{\alpha}z^{\alpha}
\end{align*}
 in $\mathbb{P}\Delta\left(0_n;1_n\right)$ and $|D(g(z))|\leq k|D(h(z))|$ for some $k\in [0,1]$, then
\begin{align*}
\mathcal{N}_f(nr):=\sum\limits_{m=0}^{\infty}\sum\limits_{|\alpha|=m}|a_{\alpha}|r^{\alpha}+\sum\limits_{m=1}^{\infty}\sum\limits_{|\alpha|=m}|b_{\alpha}|r^{\alpha}\leq 1\;\;\text{for}\;\;n{\bf{r}}\leq n{\bf{r}_0}:=\frac{1+\gamma}{3+2k+\gamma}.
\end{align*}
The radius $\frac{1+\gamma}{n(3+2k+\gamma)}$ is the best possible.
\end{theo}

\subsection{\bf {Lemma}} The following lemma is needed for the proof of Theorem \ref{Th4}.
\begin{lem}\label{Lem6} Suppose that $h(z)=\sum_{m=0}^{\infty}\sum_{|\alpha|=m} a_{\alpha}z^{\alpha}$ and $g(z)=\sum_{m=0}^{\infty}\sum_{|\alpha|=m} b_{\alpha}z^{\alpha}$ are two holomorphic functions in $\mathbb{P}\Delta\left(0_n;1_n\right)$ such that $|D(g(z))|\leq k|D(h(z))|$ for some $k\in [0,1]$.
Then \begin{align*}
	\sum_{m=1}^{\infty}\sum\limits_{|\alpha|=m} |b_{\alpha}|^2{\bf{r}}^{m}\leq k^2 \sum_{m=1}^{\infty}\sum\limits_{|\alpha|=m} |a_{\alpha}|^2{\bf{r}}^{m},	
\end{align*}	
where $z=(z_1,z_2,\ldots,z_n)\in \mathbb{P}\Delta(0_n;1_n)$ and $r=(r_1,r_2,\ldots,r_n)$ such that ${\bf{r}}=\|z\|_{\infty}$.
\end{lem}

\begin{proof}
By the given condition, we have
\begin{align}\label{L6.1}
|D(g(z))|^2\leq k^2|D(h(z))|^2.
\end{align}	
We expand $g(z)$ as the homogeneous polynomials
\begin{align*}
g(z)=\sum_{m=0}^{\infty}\sum_{|\alpha|=m} b_{\alpha}z^{\alpha}
=\sum_{m=1}^{\infty}P_m(z).
\end{align*}
where $P_{m}(z)$ is either identically zero or a homogeneous polynomial of degree $m\geq 1$. By the homogeneity of $P_m(z)$, we have
\begin{align*}
\sum\limits_{j=1}^n z_j \frac{\partial P_m(z)}{\partial z_j}=mP_m(z),
\end{align*}
where $m=1,2,\ldots$. Hence
\begin{align*}
D(g(z))=\sum\limits_{j=1}^n z_j \frac{\partial g(z)}{\partial z_j}=\sum\limits_{m=1}^{\infty} mP_m(z).
\end{align*}
For any multi-index $\nu=(\nu_1,\nu_2,\ldots,\nu_n)$, we have
\begin{align}\label{L6.2}
\int\limits_{0}^{2\pi}\ldots \int\limits_0^{2\pi} (e^{i\theta_1})^{\nu_1}\ldots (e^{i\theta_n})^{\nu_n}d\theta_1 \ldots d\theta_n=
\begin{cases}
0,& \nu\neq (0,0,\ldots,0),\\[2ex]
(2\pi)^n,& \nu=(0,0,\ldots,0).
\end{cases}
\end{align}

We set $z_j=r_je^{i\theta_j}\;(0\leq\theta_j\leq 2\pi)$, where $0<r_j<1$, $j=1,2,\ldots,n$ and consider the integration of $\left|D\left(h\left(r_1e^{i\theta_1},\ldots,r_ne^{i\theta_n}\right)\right)\right|^2$. Now using \eqref{L6.2}, we obtain
\begin{align}\label{L6.3}
\sum_{m=1}^{\infty}m^2\sum\limits_{|\alpha|=m} |a_{\alpha}|^2r^{2\alpha}=\frac{1}{(2\pi)^n} \int_{0}^{2\pi} \cdots \int_{0}^{2\pi} \left|D\left(h\left(r_1e^{i\theta_1},\ldots,r_ne^{i\theta_n}\right)\right)\right|^2 \, d\theta_1 \cdots d\theta_n.
\end{align}

Similarly, we have
\begin{align}\label{L6.4}
	\sum_{m=1}^{\infty}m^2\sum\limits_{|\alpha|=m} |b_{\alpha}|^2r^{2\alpha}=\frac{1}{(2\pi)^n} \int_{0}^{2\pi} \cdots \int_{0}^{2\pi} \left|D\left(g\left(r_1e^{i\theta_1},\ldots,r_ne^{i\theta_n}\right)\right)\right|^2 \, d\theta_1 \cdots d\theta_n.
\end{align}

Now using \eqref{L6.1}, \eqref{L6.3} and \eqref{L6.4}, we deduce that
\begin{align*}
\sum_{m=1}^{\infty}m^2\sum\limits_{|\alpha|=m} |b_{\alpha}|^2r^{2\alpha}\leq k^2 \sum_{m=1}^{\infty}m^2\sum\limits_{|\alpha|=m} |a_{\alpha}|^2r^{2\alpha}.	
\end{align*}

Letting $r_j\to {\bf{r}}$, where $j=1,2,\ldots,n$, we get
\begin{align}\label{L6.5}
	\sum_{m=1}^{\infty}m^2\sum\limits_{|\alpha|=m} |b_{\alpha}|^2{\bf{r}}^{2m-1}\leq k^2 \sum_{m=1}^{\infty}m^2\sum\limits_{|\alpha|=m} |a_{\alpha}|^2{\bf{r}}^{2m-1}.	
\end{align}

We integrate \eqref{L6.5} with respect to ${\bf{r}}$ and obtain
\begin{align*}
	\sum_{m=1}^{\infty}m\sum\limits_{|\alpha|=m} |b_{\alpha}|^2{\bf{r}}^{2m}\leq k^2 \sum_{m=1}^{\infty}m\sum\limits_{|\alpha|=m} |a_{\alpha}|^2{\bf{r}}^{2m},	
\end{align*}
i.e.,
\begin{align}\label{L6.6}
	\sum_{m=1}^{\infty}m\sum\limits_{|\alpha|=m} |b_{\alpha}|^2({\bf{r}}^2)^{m-1}\leq k^2 \sum_{m=1}^{\infty}m\sum\limits_{|\alpha|=m} |a_{\alpha}|^2({\bf{r}}^2)^{m-1}.	
\end{align}
We integrate \eqref{L6.6} once more and obtain
\begin{align*}
	\sum_{m=1}^{\infty}\sum\limits_{|\alpha|=m} |b_{\alpha}|^2{\bf{r}}^{m}\leq k^2 \sum_{m=1}^{\infty}\sum\limits_{|\alpha|=m} |a_{\alpha}|^2{\bf{r}}^{m}.	
\end{align*}
\end{proof}

\subsection{\bf {Proof of Theorem \ref{Th4}}}
\begin{proof}
	By the given condition $h$ is holomorphic in $\mathbb{P}\Delta\left(a(\gamma);r(\gamma)\right)$ such that $|h(z)|\leq 1$ for all $z\in \mathbb{P}\Delta\left(a(\gamma);r(\gamma)\right)$. Now by Lemma \ref{Lem4} we have
\begin{align}\label{Th-4.1}
|a_{\alpha}|\leq \frac{1-|a_0|^2}{1+\gamma}	
	\end{align}
	for multi-index $\alpha=(\alpha_1,\alpha_2,\ldots,\alpha_n)$ such that $|\alpha|\geq 1$. Using \eqref{Th-4.1}, we find that
\begin{align}\label{Th-4.2}
\sum\limits_{m=1}^{\infty}\sum\limits_{|\alpha|=m}|a_{\alpha}|^2{\bf{r}}^m\leq \left(\frac{1-|a_0|^2}{1+\gamma}\right)^2\sum\limits_{m=1}^{\infty}(n{\bf{r}})^m=\left(\frac{1-|a_0|^2}{1+\gamma}\right)^2\frac{n{\bf{r}}}{1-n{\bf{r}}}.
\end{align}

On the other hand by Lemma \ref{Lem6}, we have 
\begin{align}\label{Th-4.3}
	\sum_{m=1}^{\infty}\sum\limits_{|\alpha|=m} |b_{\alpha}|^2{\bf{r}}^{m}\leq k^2 \sum_{m=1}^{\infty}\sum\limits_{|\alpha|=m} |a_{\alpha}|^2{\bf{r}}^{m}.	
\end{align}

Applying \eqref{Th-4.2} and \eqref{Th-4.3}, we see that
\begin{align}\label{Th-4.4}
 \sum\limits_{m=1}^{\infty}\sum\limits_{|\alpha|=m}|b_{\alpha}|{\bf{r}}^m\leq& \sqrt{\sum\limits_{m=1}^{\infty}\sum\limits_{|\alpha|=m}|b_{\alpha}|^2{\bf{r}}^m}\times \sqrt{\sum\limits_{m=1}^{\infty}\sum\limits_{|\alpha|=m}{\bf{r}}^m}\\\leq&
 k\frac{1-|a_0|^2}{1+\gamma}\frac{n{\bf{r}}}{1-n{\bf{r}}}.\nonumber
 \end{align}
 
 Let $|a_0|=a\geq 0$. Now using \eqref{Th-4.1} and \eqref{Th-4.4}, we deduce that
 \begin{align*}
 \mathcal{N}_f(n{\bf{r}})=&\sum\limits_{m=0}^{\infty}\sum\limits_{|\alpha|=m}|a_{\alpha}|r^{\alpha}+\sum\limits_{m=1}^{\infty}\sum\limits_{|\alpha|=m}|b_{\alpha}|r^{\alpha}\\\leq&
 \frac{1-a^2}{1+\gamma}\sum\limits_{m=0}^{\infty}\sum\limits_{|\alpha|=m}{\bf{r}}^{m}+\sum\limits_{m=1}^{\infty}\sum\limits_{|\alpha|=m}|b_{\alpha}|{\bf{r}}^{m}\nonumber\\\leq&
 a+(1+k) \frac{1-a^2}{1+\gamma}\frac{n{\bf{r}}}{1-n{\bf{r}}}\nonumber\\=&
 1-\frac{1-a}{(1+\gamma)(1-n{\bf{r}})}\left(1+\gamma-n{\bf{r}}(1+\gamma+(1+a)(1+k))\right)\leq 1
 \end{align*}
 whenever $n{\bf{r}}\leq n{\bf{r}_0}(a)$, where
 \begin{align*}
 n{\bf{r}_0}(a)=\frac{1+\gamma}{1+\gamma+(1+a)(1+k)}.
 \end{align*}
 Since $a$ can be chosen arbitrarily close to $1$,
 it follows that $\mathcal{N}_f(n{\bf{r}})\leq 1$ provided that 
 \begin{align*}
 n{\bf{r}}\leq n{\bf{r}_0}(1)=\frac{1+\gamma}{3+2k+\gamma}.
 \end{align*}
 
 \medskip
 To prove the sharpness, we consider we consider the function $f_0=h_0+\ol{g_0}$ in $\mathbb{P}\Delta\left(a(\gamma);r(\gamma)\right)$, where
 \begin{align*}
 h_0(z)=\frac{a-\gamma-(1-\gamma)(z_1+z_2+\ldots+z_n)/n}{1-a\gamma-a(1-\gamma)(z_1+z_2+\ldots+z_n)/n}=A_0-\sum\limits_{m=1}^{\infty}A_m\left(\frac{z_1+\ldots+z_n}{n}\right)^m,
 \end{align*}
 where $a\in(0,1)$, $A_m(m\geq 0)$ are given by \eqref{Th-1.2} and $g_0=k\lambda(h_0(z)-A_0).$.
 
 Now for the point $z=(\tilde z,\tilde z,\ldots,\tilde z)$, we have 
 $h_0(z)=A_0-\sum_{m=1}^{\infty}A_m\tilde z^m$ and $g_0(z)=-k\lambda \sum_{m=1}^{\infty}A_m\tilde z^m$.
 For $\tilde z=nr$, where $\left|nr+\frac{\gamma}{1-\gamma}\right|<\frac{1}{1-\gamma}$, we find that
 \begin{align*}
 \mathcal{N}_{f_0}(n{\bf{r}})=&A_0+(1+k\lambda)\sum\limits_{m=1}^{\infty}A_m(nr)^m\\=&
 \frac{a-\gamma}{1-a\gamma}+(1+k\lambda)\frac{1-a^2}{a(1-a\gamma)}\sum\limits_{m=1}^{\infty}\left(\frac{a(1-\gamma)}{1-a\gamma}\right)^m(nr)^m\\=&
 \frac{a-\gamma}{1-a\gamma}+(1+k\lambda)\frac{1-a^2}{1-a\gamma}\frac{(1-\gamma)nr}{1-a\gamma-a(1-\gamma)nr}=
 1-\frac{1-a}{1-a\gamma}\Phi(nr),
 \end{align*}
 where
 \begin{align*}
 \Phi(nr)=1+\gamma-(1+k\lambda)\frac{(1+a)(1-\gamma)nr}{1-a\gamma-a(1-\gamma)nr},
 \end{align*}

 \begin{figure}[H]
 	\centering
 	\begin{tikzpicture}
 	\begin{axis}[
 		width=0.72\textwidth,
 		height=0.50\textwidth,
 		xlabel={$n{\bf r}$},
 		ylabel={$\Phi(n{\bf r})$},
 		xmin=0,
 		xmax=0.45,
 		ymin=-0.00,
 		ymax=1.40,
 		axis lines=left,
 		grid=both,
 		grid style={line width=.1pt,draw=gray!25},
 		major grid style={line width=.25pt,draw=gray!50},
 		samples=300,
 		domain=0:0.38,
 		clip=true,
 		restrict y to domain=-0.10:1.22,
 		legend pos=north east,
 		legend style={
 			draw=black,
 			fill=white,
 			font=\small,
 			cells={anchor=west},
 			/tikz/every even column/.append style={column sep=6pt}
 		},
 		legend image post style={line width=2pt},
 		axis line style={
 			very thick,
 			->,
 			>=Latex
 		},
 		tick style={black,thick},
 		]

 	\addplot[
 	blue!80!black,
 	line width=.5pt,
 	densely dashed,
 	line cap=round
 	]
 	{1+0.2 - (1+1*1)*(1+0.5)*(1-0.2)*x/(1-0.5*0.2-0.5*(1-0.2)*x)};
 	\addlegendentry{$a=0.50$}
 	
 	\addplot[
 	teal!80!black,
 	line width=.5pt,
 	densely dashdotted,
 	line cap=round
 	]
 	{1+0.2 - (1+1*1)*(1+0.8)*(1-0.2)*x/(1-0.8*0.2-0.8*(1-0.2)*x)};
 	\addlegendentry{$a=0.80$}
 	
 	\addplot[
 	orange!95!black,
 	line width=.5pt,
 	densely dotted,
 	line cap=round
 	]
 	{1+0.2 - (1+1*1)*(1+0.95)*(1-0.2)*x/(1-0.95*0.2-0.95*(1-0.2)*x)};
 	\addlegendentry{$a=0.95$}
 	
 	\addplot[
 	red!90!black,
 	line width=.8pt,
 	solid,
 	line cap=round
 	]
 	{1+0.2 - (1+1*1)*(1+0.995)*(1-0.2)*x/(1-0.995*0.2-0.995*(1-0.2)*x)};
 	\addlegendentry{$a\rightarrow1$}

 			\addplot[black, thin] coordinates {(0,0) (0.55,0)};

 			\addplot[
 			black,
 			dash pattern=on 5pt off 2pt,
 			line width=2pt
 			]coordinates {(0.2308,-0.15) (0.2308,1.25)};
 			\node[below] at (axis cs:0.2308,-0.15) {\small $n{\bf r}_0(1)=\dfrac{1+\gamma}{3+2k+\gamma}$};
 			
 			\node[circle,fill=red,inner sep=1.3pt] at (axis cs:0.2308,0.0107) {};
 			
 		\end{axis}
 	\end{tikzpicture}
 	\caption{Behaviour of $\Phi(n{\bf r})$ for $\gamma=0.2$, $k=1$, $\lambda\to 1$, and $a=0.50,\,0.80,\,0.95,\,a\to 1$. Each curve is strictly decreasing on $[0,1)$; as $a\to 1$ the curve touches $0$ precisely at $n{\bf r}_0(1)=\frac{1+\gamma}{3+2k+\gamma}$, confirming that the bound $\mathcal{N}_{f_0}(n{\bf r})\le 1$ fails exactly beyond this radius and hence that $n{\bf r}_0(1)$ is sharp.}
 	\label{fig:Phi-sharpness}
 \end{figure}
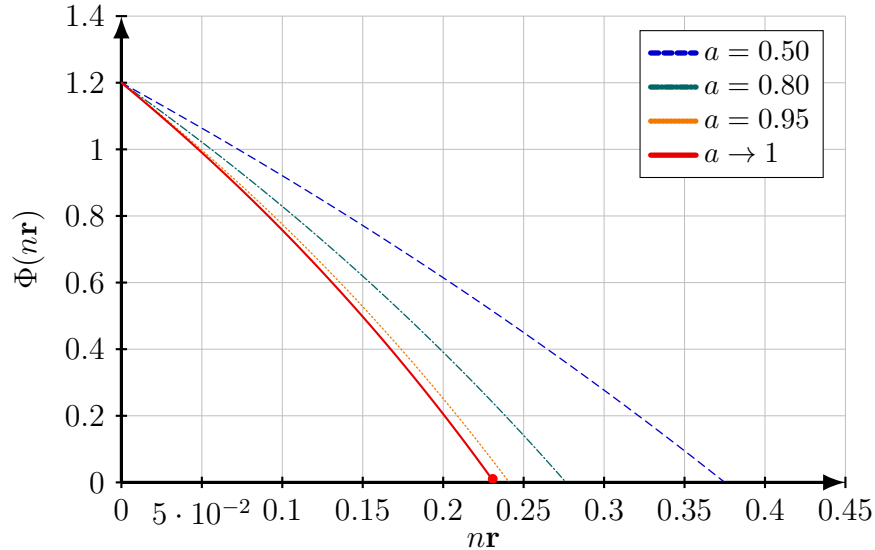
 which is strictly decreasing for $nr\in[0,1)$ and so for \begin{align*}
 nr>nr_0(1)=\frac{1+\gamma}{3+2k+\gamma},
 \end{align*}
  we have
 \begin{align*}
 \Phi(nr)<\Phi(nr_0(1))=1+\gamma-(1+k\lambda)\frac{(1+a)(1-\gamma^2)}{(1-a\gamma)(3+2k+\gamma)-a(1-\gamma^2)}\to 0
 \end{align*}
 as $a$ and $\lambda$ tend to $1$. Therefore $\mathcal{N}_{f_{0}}(nr)>1$ for $nr>nr_0(1)=\frac{1+\gamma}{3+2k+\gamma}$. This shows that the number $\frac{1+\gamma}{n(3+2k+\gamma)}$ cannot be replaced by a larger quantity. Hence the radius $\frac{1+\gamma}{3+2k+\gamma}$ is sharp.
\end{proof}

\vspace{5mm}

\noindent\textbf{Conflict of interest:} The authors declare that there is no conflict  of interest regarding the publication of this paper.\vspace{1.2mm}

\noindent {\bf Funding:} Not Applicable.\vspace{1.2mm}

\noindent\textbf{Data availability statement:}  Data sharing not applicable to this article as no datasets were generated or analysed during the current study.\vspace{1.2mm}

\noindent {\bf Authors' contributions:} All the authors have equal contributions in preparation of the manuscript.

\end{document}